\documentclass{jgcc}
\pdfoutput=1
\usepackage{lastpage}
\jgccdoi{17}{1}{3}{13561}
\jgccheading{}{\pageref{LastPage}}{}{}{May~8,~2024}{May~23,~2025}{}  
\usepackage[utf8]{inputenc}
\usepackage{amsthm}
\usepackage{amsmath}
\usepackage{amsfonts}
\usepackage{ amssymb }
\usepackage[title]{appendix}
\usepackage{adjustbox}
\usepackage{multirow}
\usepackage{multicol}
\usepackage{color}
\usepackage{array}
\usepackage{graphicx}
\usepackage{hyperref}
\usepackage{mathtools}
\usepackage{xcolor}
\usepackage{tikz}
\usetikzlibrary{arrows, arrows.meta}
\usepackage{changepage}
\usepackage{thm-restate}

\newcommand{\comm}[1]{}
\newcommand{\Z}{\mathbb{Z}}

\newcommand{\e}{\varepsilon}
\newcommand{\X}{X^{\ast}}
\newcommand{\LR}{\xleftrightarrow{\phi}}
\newcommand{\Pos}{\mathrm{Pos}}
\newcommand{\Neg}{\mathrm{Neg}}

\newcommand{\Mod}[1]{\ (\mathrm{mod}\ #1)}
\usepackage[capitalise]{cleveref}

\title{Twisted conjugacy in dihedral Artin groups I: Torus Knot Groups}
\author[Gemma Crowe]{Gemma Crowe}
\address{Department of Mathematics, University of Manchester M13 9PL, UK and the Heilbronn Institute for Mathematical Research, Bristol, UK}
\email{gemma.crowe@manchester.ac.uk}

\date{}
\begin{document}

\begin{abstract}
In this paper we provide an alternative solution to a result by (Juh\'{a}sz 2011), that the twisted conjugacy problem for odd dihedral Artin groups is solvable, where an odd dihedral Artin group is a group with presentation $G(m) = \langle a,b \mid {}_{m}(a,b) =  {}_{m}(b,a) \rangle$, where $m\geq 3$ is odd, and $_{m}(a,b)$ is the word $abab \dots$ of length $m$. Our solution provides an implementable linear time algorithm, by considering an alternative group presentation to that of a torus knot group, and working with geodesic normal forms. An application of this result is that the conjugacy problem is solvable in extensions of odd dihedral Artin groups.
\end{abstract}
\maketitle

\section{Introduction}
Over a century ago, Dehn defined his famous decision problems, which have been studied across a wide range of groups. One of these is the \emph{conjugacy problem}, which asks if there exists an algorithm to determine if two elements in a group, given by words on the generators, are conjugate. This decision problem has been generalised to include an algebraic property known as \emph{twisted conjugacy}. 

For a finitely generated group $G$ with inverse-closed generating set $X$, we say two elements $u,v \in G$ are \emph{twisted conjugate}, by some automorphism $\phi \in \mathrm{Aut}(G)$, if there exists an element $w \in G$ such that $v = \phi^{-1}(w)uw$. The \emph{twisted conjugacy problem} (TCP) asks whether there exists an algorithm to determine if two elements, given as words over $X$, are twisted conjugate by some automorphism $\phi \in \mathrm{Aut}(G)$. Whilst the twisted conjugacy problem is a relatively new decision problem, in comparison to the conjugacy problem, several important results have been established (see \cite{bogopolski_conjugacy_2006, burillo_conjugacy_2016, cox_twisted_2017, gonzalez-meneses_twisted_2014}). Note that a positive solution to the twisted conjugacy problem implies a positive solution to the conjugacy problem, however the converse is not necessarily true \cite[Corollary 4.9]{bogopolski_orbit_2009}. 

In a series of two papers, we solve the twisted conjugacy problem for dihedral Artin groups. Artin groups are defined by a simple graph with edge labelling from $\Z_{\geq 2}$. Despite their long history of study, there are still open questions, such as the decidability of the word problem, which are not known for all Artin groups. A typical approach to understand these open problems is to consider subclasses of Artin groups, which are often defined using conditions on the defining graph. Dihedral Artin groups are one such class, and are defined by a simple graph with precisely two vertices and one edge, that is, they have a group presentation $G(m) = \langle a,b \mid {}_{m}(a,b) =  {}_{m}(b,a) \rangle$, where $m\geq 3$ and $_{m}(a,b)$ is the word $abab \dots$ of length $m$. 

These one-relator, torsion free groups are appealing to study from a computational perspective, and recent studies have found interesting results such as decidability of equations \cite{ciobanu_equations_2020}, and solvable conjugacy problem in linear time \cite{Holt2015}. We note here that the twisted conjugacy problem for dihedral Artin groups was already known to be solvable by Juh\'{a}sz \cite{juhasz_twisted_2011}, with respect to length-preserving automorphisms. To solve the twisted conjugacy problem for dihedral Artin groups with respect to all automorphisms, we consider when the edge labelling of our defining graph is either odd or even. In this first paper, we focus on odd dihedral Artin groups and further improve on Juh\'{a}sz's result, by constructing an algorithm to check if two words, representing group elements, are twisted conjugate. Moreover, this algorithm has linear time complexity, based on the length of the input words.

\begin{restatable*}{thm}{TCP}\label{thm:main result TCP solvable}
    The twisted conjugacy problem $\mathrm{TCP}(G(m))$ for dihedral Artin groups, where $m$ is odd, $m \geq 3$, is solvable in linear time.
\end{restatable*}
The key step is to utilise the fact that any odd dihedral Artin group $G(m)$ is isomorphic to a \emph{torus knot group}. In particular, $G(m)$ is isomorphic to the group with presentation $\langle x,y \mid x^{2}=y^{m} \rangle$, where $m$ is odd. With this presentation, we find that the automorphism group of $G(m)$ is finite of order two. Moreover, we can use a geodesic normal form derived by Fujii \cite{Fujii2018}, to construct a linear time algorithm to solve the $\mathrm{TCP}(G(m))$.

As an application of Theorem \ref{thm:main result TCP solvable}, we consider the conjugacy problem in extensions of odd dihedral Artin groups, using a criteria from \cite{bogopolski_orbit_2009} (see Theorem \ref{thm:orbit conditions}). This criteria requires understanding \emph{orbit decidability} for subgroups $A \leq \mathrm{Aut}(G)$, for some group $G$. The \emph{orbit decidability problem} asks whether we can determine if for two elements $u,v \in G$, there exists an automorphism $\phi \in A$ such that $v$ is conjugate to $\phi(u)$. To apply the criteria from \cite{bogopolski_orbit_2009}, we require what is known as the `action subgroup' of odd dihedral Artin groups to be orbit decidable. For odd dihedral Artin groups, we prove a stronger statement, by showing that all subgroups of the automorphism group of odd dihedral Artin groups are orbit decidable. 

\begin{restatable*}{thm}{orbit}\label{thm:orbit decid}
    Every finitely generated subgroup $A \leq \mathrm{Aut}(G(m))$, when $m$ is odd, is orbit decidable.
\end{restatable*}
Combined with Theorem \ref{thm:main result TCP solvable}, we can find new examples of groups with solvable conjugacy problem. 

\begin{restatable*}{thm}{extension}\label{thm:extension conj problem}
    Let $G = G(m) \rtimes H$ be an extension of an odd dihedral Artin group by a finitely generated group $H$ which satisfies conditions $(ii)$ and $(iii)$ from Theorem \ref{thm:orbit conditions} (e.g. let $H$ be torsion-free hyperbolic). Then $G$ has decidable conjugacy problem. 
\end{restatable*}
We note that the conjugacy problem is solvable for virtual dihedral Artin groups \cite[Corollary 3.2]{ciobanu_equations_2020}, which proves Theorem \ref{thm:extension conj problem} in the case where $H$ is finite.

The structure of this paper is as follows. After providing necessary details on decision problems, twisted conjugacy and odd dihedral Artin groups in \cref{sec:prelims}, we construct a linear time solution with respect to outer automorphisms in \cref{sec:algorithm}. This is then extended in \cref{sec:full algorithm} to solve the TCP for odd dihedral Artin groups. Finally in \cref{sec:extension}, we study orbit decidability, which allows us to use the criteria from \cite{bogopolski_orbit_2009} to find new examples of group extensions of odd dihedral Artin groups with solvable conjugacy problem. 

\section{Preliminaries}\label{sec:prelims}
All groups in this paper are finitely generated, and all finite generating sets are inverse-closed. For a subset $S$ of a group, we let $S^{\pm} = S \cup S^{-1}$, where $S^{-1} = \{s^{-1} \mid s \in S \}$. In certain cases, group presentations will be stated using generating sets which are not inverse-closed, and when we pass to the relevant inverse-closed generating set, we won't explicitly mention the corresponding change in the group presentation. 
\subsection{Decision problems and twisted conjugacy}
Let $X$ be a finite set, and let $X^{\ast}$ be the set of all finite words over $X$. For a group $G$ generated by $X$, we use $u=v$ to denote equality of words in $X^{\ast}$, and $u =_{G} v$ to denote equality of the group elements represented by $u$ and $v$. We let $l(w)$ denote the word length of $w$ over $X$. For a group element $g \in G$, we define the \emph{length} of $g$, denoted $|g|_{X}$, to be the length of a shortest representative word for the element $g$ over $X$. A word $w \in X^{\ast}$ is \emph{geodesic} if $l(w) = |\pi(w)|_{X}$, where $\pi \colon X^{\ast} \rightarrow G$ is the natural projection. If there exists a unique word $w$ of minimal length representing $g$, then we say $w$ is a \emph{unique geodesic}. Otherwise, $w$ is a non-unique geodesic. We let $\varepsilon$ denote the empty word over $X$ representing the identity element of $G$.    

\begin{defi}
    Let $G = \langle X \rangle$. The \emph{word problem} for $G$, denoted $\mathrm{WP}(G)$, takes as input a word $w \in X^{\ast}$, and decides whether it represents the trivial element of $G$. The \emph{conjugacy problem} for $G$, denotes $\mathrm{CP}(G)$, takes as input two words $u,v \in X^{\ast}$, and decides whether they represent conjugate elements in $G$. We write $u \sim v$ when $u$ and $v$ represent conjugate elements in $G$. 
\end{defi}

\begin{defi}
    Let $G = \langle X \rangle$, let $u,v \in X^{\ast}$, and let $\phi \in \mathrm{Aut}(G)$ be an automorphism of $G$.
    \begin{enumerate}
        \item We say $u$ and $v$ are \emph{$\phi$-twisted conjugate}, denoted $u \sim_{\phi} v$, if there exists an element $w \in G$ such that $v =_{G} \phi(w)^{-1}uw$.
        \item The \emph{$\phi$-twisted conjugacy problem} for $G$, denoted $\mathrm{TCP}_{\phi}(G)$, takes as input two words $u,v \in X^{\ast}$, and decides whether they represent groups elements which are $\phi$-twisted conjugate to each other in $G$.
        \item The (uniform) \emph{twisted conjugacy problem} for $G$, denoted $\mathrm{TCP}(G)$, takes as input two words $u,v \in X^{\ast}$ and $\phi \in \mathrm{Aut}(G)$, and decides whether $u$ and $v$ represent groups elements which are $\phi$-twisted conjugate in $G$.
    \end{enumerate}
\end{defi}
A solution to the $\mathrm{TCP}(G)$ implies a solution to $\mathrm{TCP}_{\phi}(G)$ for all $\phi \in \mathrm{Aut}(G)$, and therefore a solution to the $\mathrm{CP}(G)$ and the $\mathrm{WP}(G)$. We first observe that to solve the $\mathrm{TCP}(G)$, we can study the outer automorphism group, rather than the automorphism group as a whole. This was originally noted in \cite{gonzalez-meneses_twisted_2014}.

\begin{rem}\label{rmk:outer autos enough}
    Let $\psi \in \mathrm{Aut}(G)$, and let $[\psi]$ denote the equivalence class of $\psi$ in $\mathrm{Out}(G)$. We can write $\psi$ in the form $\psi = \iota_{g}\phi$, where $\iota_{g} \in \mathrm{Inn}(G)$ denotes conjugation by $g$, and $\phi \in [\psi]$. If $[\psi]$ is trivial, then the $\mathrm{TCP}_{\psi}(G)$ is equivalent to solving the $\mathrm{CP}(G)$ by the following relations:
\[ v =_{G} \psi(w)^{-1}uw \Leftrightarrow v =_{G} g^{-1}w^{-1}guw \Leftrightarrow gv =_{G} w^{-1}(gu)w,
\]
for some fixed $g \in G$. Otherwise, we have
\[ v =_{G} (\iota_{g}\phi)(w)^{-1}uw \Leftrightarrow v =_{G} g^{-1}\phi(w)^{-1}guw \Leftrightarrow gv =_{G} \phi(w)^{-1}(gu)w,
\]
for some fixed $g \in G$. In particular, if $[\psi]$ is non-trivial, then the $\mathrm{TCP}_{\psi}(G)$ is equivalent to solving the $\mathrm{TCP}_{\phi}(G)$ for some fixed $\phi \in [\psi]$. 
\end{rem}

\comm{
Since we are interested in finite order automorphisms, the following Lemma implies that it is enough to consider finite order outer automorphisms.
\begin{lem}
    If $\phi \in \mathrm{Out}(G)$ has infinite order, then $\psi \circ \phi \in \mathrm{Aut}(G)$ has infinite order.
\end{lem}
In studying finite order outer automorphisms, we not only capture all cases of finite order automorphisms, but also some infinite order automorphisms.}

\begin{defi}
   Let $G = \langle X \rangle$ be a group. We say $\phi \in \mathrm{Aut}(G)$ is:
\begin{enumerate}
    \item[(i)] \emph{length-preserving} if $|\pi(w)| = |\phi(\pi(w))|$ for any word $w \in X^{\ast}$, and
    \item[(ii)] \emph{non-length preserving} otherwise.
\end{enumerate}
\end{defi}

\subsection{Odd dihedral Artin groups}
\begin{defi}
    Let $m \in \Z_{>1}$. The dihedral Artin group is the group defined by the following presentation:
    \begin{equation}\label{eq:all groups}
    G(m) = \langle a,b \mid {}_{m}(a,b) =  {}_{m}(b,a) \rangle,
\end{equation}
     where $_{m}(a,b)$ is the word $abab \dots$ of length $m$.
\end{defi}
We first note the following result, which solves the $\mathrm{TCP}_{\phi}(G(m))$ in the case where the automorphism $\phi \in \mathrm{Aut}(G(m))$ is trivial.
\begin{thm}\label{thm:conj prob decid DAm}\cite[Proposition 3.1]{Holt2015}
    The conjugacy problem is solvable in dihedral Artin groups in linear time.
\end{thm}
If $m=2$, then $G(m)$ is the free abelian group, and if $m=3$, then $G(m)$ is the braid group $B_{3}$. We note that the twisted conjugacy problem has been solved in both of these cases \cite{bogopolski_orbit_2009, gonzalez-meneses_twisted_2014}. For the remainder of this paper, we will assume $m \geq 3$ is odd. To solve the TCP for $G(m)$, we will work with a different generating set, which we now define. 

\begin{defi}\label{defn: gen set free product}
    Let $m = 2k+1$, where $k \in \Z_{\geq 1}$. We define
    \begin{equation}\label{eqn: m odd}
    P = \langle x,y \mid x^{2} = y^{m} \rangle
\end{equation}
    to be the presentation of a \emph{torus-knot group}. For notation we let $X = \{x,y\}^{\pm}$, which we refer to as the \emph{free product generating set}.
\end{defi}

\begin{lem}
    When $m \geq 3$ is odd, the presentations given in \cref{eq:all groups} and \cref{eqn: m odd} represent isomorphic groups.  
\end{lem}
\begin{proof}
    We apply a sequence of Tietze transformations to $G(m)$ to obtain $P$ as follows. First, set $x = {}_{m}(a,b)$ and $y = ab$, which gives us 
    \[ G(m) \cong \langle a,b,x,y \mid {}_{m}(a,b) = {}_{m}(b,a), x = {}_{m}(a,b), y = ab \rangle.
    \]
    Since $x = by^{k} = y^{k}a$, then $x^{2} = y^{k}aby^{k} = y^{2k+1}$, and so our relations can be rewritten as 
    \[ G(m) \cong \langle a,b,x,y \mid x = {}_{m}(a,b) = {}_{m}(b,a), y = ab, x^{2} = y^{2k+1}  \rangle.
    \]
    Next we remove the generator $a$, by rewriting our relation $y = ab$ as $a = yb^{-1}$. Removing redundant relations gives us 
    \[ G(m) \cong \langle b,x,y \mid x = y^{k+1}b^{-1} = by^{k}, x^{2} = y^{2k+1} \rangle.
    \]
    Finally we remove the generator $b$ in a similar way by rewriting the relation $x=by^{k}$, and removing redundant relations. This leaves us with $G(m) \cong P$ as required. 
\end{proof}
Before considering twisted conjugacy, we first need to understand the outer automorphisms of this group.
\begin{thm}\label{thm:odd autos}\cite[Theorem C]{gilbert_tree_2000}
    Let $G(m)$ be defined as in \cref{eqn: m odd}. Then $\mathrm{Out}(G(m)) \cong C_{2}$. 
\end{thm}
In particular, $\mathrm{Aut}(G(m)) = \mathrm{Inn}(G(m)) \sqcup \mathrm{Inn}(G(m))\cdot \phi$ where $\phi$ inverts all generators, i.e. 
\begin{equation}\label{eqn: odd phi map}
    \phi \colon x \mapsto x^{-1}, \; y \mapsto y^{-1}.
\end{equation}
For the remainder of this section, we assume $\phi \in \mathrm{Aut}(G(m))$ as in \cref{eqn: odd phi map}. We let $\mathsf{rev}(w)$ denote the word $w \in X^{\ast}$ written in reverse, and note that for any word $w \in X^{\ast}$, $\phi(w)^{-1} = \mathsf{rev}(w)$. Indeed if $w = w^{a_{1}}_{i_{1}}\dots w^{a_{n}}_{i_{n}} \in X^{\ast}$, where each $w_{i_{j}} \in X$ and $a_{j} \in \Z_{\neq 0}$ $(1 \leq j \leq n)$, then $\phi(w) = w^{-a_{1}}_{i_{1}}\dots w^{-a_{n}}_{i_{n}}$, and so $\phi(w)^{-1} = w^{a_{n}}_{i_{n}} \dots w^{a_{1}}_{i_{1}} = \mathsf{rev}(w)$. Therefore checking if two elements are twisted conjugate reduces to checking if $v =_{G} \mathsf{rev}(w)uw$. 

We now provide a set of geodesic normal forms for $G(m)$. These were originally derived from \cite{Fujii2018}, and further details can be found in \cite{ciob_conjgrowth} and \cref{appendix}. 
\begin{prop}\label{prop:geos}
    Any element $g \in G(m)$ can be represented by a geodesic word $u = x^{a_{1}}y^{b_{1}}\dots y^{b_{\tau}}\Delta^{c}$, where $\Delta=_{G} x^2 =_{G} y^{2k+1}$, with conditions given in \cref{tab:geo forms odd}.
\end{prop}
We refer to $\Delta$ as the \emph{Garside element}, and note that any power of $\Delta$ is central.

\begin{defi}\label{defn:garside free}
    Let $u = x^{a_{1}}y^{b_{1}}\dots y^{b_{\tau}}\Delta^{c} \in X^{\ast}$ be geodesic. We say $u$ is \emph{Garside-free} if $c = 0$. 
\end{defi}
Let $u \in X^{\ast}$ be a geodesic, where $l(u) > 1$. We can write $u$ in the form $u = u_{1}u_{2}$ for some non-empty geodesic words $u_{1}, u_{2} \in X^{\ast}$. We say $u_{1}$ is a \emph{proper prefix} of $u$, and $u_{2}$ is a \emph{proper suffix} of $u$. We let $\mathcal{P}(u)$ and $\mathcal{S}(u)$ be the set of all possible proper prefixes and suffixes of $u$ respectively. With these definitions, we can define an equivalent notion of cyclic permutations, with respect to twisted conjugacy.
\begin{defi}\label{defn: odd cyclic perms}    
    Let $u \in X^{\ast}$ be a geodesic, where $l(u) > 1$. Let $u =u_{1}u_{2} $ for some $u_{1} \in \mathcal{P}(u), \; u_{2} \in \mathcal{S}(u)$. Let $\phi \in \mathrm{Aut}(G(m))$ as defined in \cref{eqn: odd phi map}. We define a \emph{$\phi$-cyclic permutation of a prefix} of $u$ to be the following operation on $u$:
    \[ u = u_{1}u_{2} \mapsto u_{2}\phi(u_{1}).
    \]
    Similarly we define a \emph{$\phi$-cyclic permutation of a suffix} of $u$ as
    \[ u = u_{1}u_{2} \mapsto \phi(u_{2})u_{1}.     \]
\end{defi}
For brevity, we will use $\xleftrightarrow{\phi}$ to denote a $\phi$-cyclic permutation of either a prefix or suffix. We note that any $\phi$-cyclic permutation can be reversed as follows:
\[ u_{1}u_{2} \mapsto u_{2}\phi(u_{1}) \mapsto \phi(\phi(u_{1}))u_{2} = u_{1}u_{2}.
\]
If two geodesics $u,v \in X^{\ast}$ are related by a $\phi$-cyclic permutation, then $u \sim_{\phi} v$. 

\section{Algorithm for Outer Automorphisms}\label{sec:algorithm}
Our main goal is to solve the twisted conjugacy problem $\mathrm{TCP}_{\phi}(G(m))$, where $\phi \in \mathrm{Aut}(G(m))$ as in \cref{eqn: odd phi map}. We will then generalise this algorithm, using Remark \ref{rmk:outer autos enough}, to solve the twisted conjugacy problem $\mathrm{TCP}(G(m))$.

We will first show that when $\phi \in \mathrm{Aut}(G(m))$ as in \cref{eqn: odd phi map}, then every geodesic is twisted conjugate to a minimal length twisted conjugacy representative which is Garside free. Moreover, for two minimal representatives which are twisted conjugate, we will show that they are related by a finite sequence of $\phi$-cyclic permutations and equivalent geodesics. Since $\phi$ is length-preserving, this will lead to our basic algorithm, which we later improve to give a linear time algorithm.
\subsection{Garside free representatives}\label{sec:simple}
\begin{defi}\label{defn:odd Type 3}
    Let $u \in X^{\ast}$ be geodesic. We say $u$ is of Type (3) if $u$ is not of Type (1) or (2) (see \cref{tab:geo forms odd}). In particular, Type (3) geodesics are Garside free (recall Definition \ref{defn:garside free}). We let $\overline{(3)} \subset (3)$ denote the set of all words $u \in (3)$, such that $u$ starts and ends with opposite letters.
\end{defi}
Our first step is to prove the following result, which shows that any geodesic $u \in X^{\ast}$ can be transformed to a geodesic $u' \in X^{\ast}$, which is Garside free, such that $u \sim_{\phi} u'$. 
\begin{prop}\label{prop:rewrite}
    Let $u \in \X$ be a geodesic. Then there exists a geodesic $u' \in \overline{(3)}$, such that $u \sim_{\phi} u'$ and $l(u') \leq l(u)$. 
\end{prop}
We first consider some simple cases.
\begin{lem}\label{cor: odd easy case}
    Let $u = x^{t}$ and $v = y^{t+mc}$, for some $t \in \Z_{\geq 0}, c \in \Z$. Then
    \begin{equation*}
        \begin{split}
        u = x^{t} \sim_{\phi}
        \begin{cases}
        \varepsilon, & t \quad \text{is even}, \\
        x^{\pm 1}, & t \quad \text{is odd},
    \end{cases}
        \end{split}
        \quad \text{and} \quad
    \begin{split}
        v = y^{t+mc} \sim_{\phi} 
        \begin{cases}
        \varepsilon, & c,t \quad \text{have same parity}, \\
        y^{\pm 1}, & c,t \quad \text{have opposite parity}.
    \end{cases}
    \end{split}
    \end{equation*}
\end{lem}
\begin{proof}
    First consider $u = x^{t}$. If $t = 2d$, for some $d \in \Z$, we can apply the following $\phi$-cyclic permutation to $u$:
    \[ u =_{G} x^{d}x^{d} \xleftrightarrow{\phi} x^{-d}x^{d} =_{G} \e.
    \]
    Similarly if $t = 2d+1$, we have
    \[ u =_{G} x^dx^{d+1} \xleftrightarrow{\phi} x^{-d}x^{d+1} =_{G} x \xleftrightarrow{\phi} x^{-1}.     \]
    Now suppose $v = y^{t+mc} =_{G.} y^{t}\Delta^{c}$.  If $c = 2d$, for some $d \in \Z$, then $v =_{G} y^{t}\Delta^{2d} \LR \Delta^{-d}y^{t}\Delta^{d} =_{G} y^{t}.$ If $t$ is even, then $y^{t} \LR y^{-t/2}y^{t/2} =_{G} \varepsilon$. Similarly if $t$ is odd, then $y^{t} \LR y^{\pm 1}$. Otherwise, if $c = 2d+1$, then $v = y^{t}\Delta^{2d+1} \LR \Delta^{-d}y^{t}\Delta \Delta^{d} =_{G} y^{t}\Delta = y^{t+m}.$ If $t$ is odd, then $t+m$ is even, and so $y^{t+m} \LR \e$. Similarly if $t$ is even, then $t+m$ is odd and so $y^{t+m} \LR y^{\pm 1}$.
\end{proof}
\comm{
\begin{cor}\label{simple1}
    Let $u = \Delta^{c}, \; v = \Delta^{d} \in X^{\ast}$ for some $c,d \in \Z$. Then $u \sim_{\phi} v \sim_{\phi} \e$.
\end{cor}}
To prove Proposition \ref{prop:rewrite}, we will construct a set of rewriting rules which can be repeatedly applied to a geodesic $u \in X^{\ast}$ to obtain an element $u' \in \overline{(3)}$. These rules are invariant under $\sim_{\phi}$, that is, after applying a sequence of these rules to $u$, we obtain a word $u' \in \overline{(3)}$ such that $u \sim_{\phi} u'$. Let $v \in \X$ such that $v$ is Garside free, let $x,y \in X$ and $\epsilon, \epsilon_{1}, \epsilon_{2} \in \{\pm 1\}$. We define our set of rewrite rules as follows.
\begin{itemize}
    \item[(R1):] $v\Delta^{c} \rightarrow \begin{cases}
        v, & c = 2d, \; d \in \Z_{\neq 0}, \\
        v\Delta, & c = 2d+1, \; d \in \Z_{\neq 0}.
    \end{cases}$
    \item[(R2):] For $c \in \{0,1\}$, $x^{\epsilon_{1}} vx^{\epsilon_{2}}\Delta^{c} \rightarrow \begin{cases}
        v\Delta^{c}, & \epsilon_{1} = \epsilon_{2}, \\
        v\Delta^{c + \epsilon}, & \epsilon_{1} \neq \epsilon_{2}.
    \end{cases}$
    \item[(R3):]  For $c \in \{0,1\}$, $y^{b_{1}}vy^{b_{n}}\Delta^{c} \rightarrow \begin{cases}
        y^{b_{1}-b_{n}}v\Delta^{c} & b_{1}-b_{n} \in [-(m-1), m-1] \cap \Z_{\neq 0}, \\ 
        v\Delta^{c}, & b_{1}=b_{n}, \\
        v\Delta^{c + \epsilon}, & |b_{1}-b_{n}| = m.
    \end{cases}$
    \item[(R4):] $x^{\epsilon_{1}}vy^{b}\Delta^{\epsilon_{2}} \rightarrow x^{-\epsilon_{1}}vy^{b}$. 
    \item[(R5):] $y^{b}vx^{\epsilon_{1}}\Delta^{\epsilon_{2}} \rightarrow y^{b}vx^{-\epsilon_{1}}$.
\end{itemize}
All rules either preserve or decrease word length.
\begin{prop}\label{prop:odd rules invariant}
    The rules (R1)-(R5) are invariant under $\sim_{\phi}$.
\end{prop}
\begin{proof}
    Each rule can be obtained by applying $\phi$-cyclic permutations and relations from $G(m)$ as follows.
    \begin{enumerate}
        \item[(R1):] If $c=2d$, then $v\Delta^{2d} \LR \Delta^{-d}v\Delta^{d} =_{G} v$. If $c = 2d+1$, then $v\Delta^{2d+1} \LR \Delta^{-d}v\Delta^{d}\Delta =_{G} v\Delta$.
        \item[(R2):] $x^{\epsilon_{1}} vx^{\epsilon_{2}}\Delta^{c} =_{G} x^{\epsilon_{1}} v\Delta^{c}x^{\epsilon_{2}} \LR x^{\epsilon_{1}-\epsilon_{2}}v\Delta^{c} $. This is equivalent to $v\Delta^{c}$ when $\epsilon_{1} = \epsilon_{2}$, and $x^{\pm 2}v\Delta^{c} =_{G} v\Delta^{c \pm 1}$ when $\epsilon_{1} \neq \epsilon_{2}$.
        \item[(R3):] $y^{b_{1}}vy^{b_{n}}\Delta^{c} =_{G} y^{b_{1}}v\Delta^{c}y^{b_{n}} \LR y^{b_{1}-b_{n}}v\Delta^{c}$. The result then follows since $y^{m} = \Delta$. 
        \item[(R4):] If $\epsilon_{1} = \epsilon_{2}$, then $x^{\epsilon_{1}}vy^{b}\Delta^{\epsilon_{2}} = x^{\epsilon_{1}}vy^{b}x^{2\epsilon_{1}}  \LR x^{\epsilon_{1} -2\epsilon_{1}}vy^{b} =_{G} x^{-\epsilon_{1}}vy^{b}$. Otherwise, if $\epsilon_{1} = -\epsilon_{2}$, then $x^{\epsilon_{1}}vy^{b}\Delta^{\epsilon_{2}}= x^{\epsilon_{1}}vy^{b}x^{-2\epsilon_{1}} =_{G} x^{-\epsilon_{1}}vy^{b}$. (R5) follows a similar proof.
    \end{enumerate}
\end{proof}
\begin{proof}[Proof of Proposition \ref{prop:rewrite}]
Assume $u \not \in \overline{(3)}$. If $u$ contains only $x$ or $y$ factors, then we obtain $u' \in \{\varepsilon, x^{\pm 1}, y^{\pm 1}\}$, such that $u \sim_{\phi} u'$, based on the conditions given in Lemma \ref{cor: odd easy case}. 

We can therefore assume $u$ contains both $x$ and $y$ factors, and so $u = u_{1}\Delta^{c}$, where $u_{1} \in (3)$ is Garside free, and $c \in \Z$. We apply the following sequence of rewrite rules, to obtain an element $u' \in \overline{(3)}$:
\begin{enumerate}
    \item If $c \neq 0$, apply (R1) to obtain an element of the form $u_1$ or $u_{1}\Delta$. 
    \item If $u_{1}$ starts and ends with the same letter, then repeatedly apply (R2)-(R3), to obtain an element of the form $u_{1}$ or $u_{1}\Delta^{\pm 1}$, where $u_{1} \in \overline{(3)}$. If, after any application of (R2)-(R3), we obtain an element of the form $u_{1}\Delta^{\pm 2}$, then return to Step 1 and apply (R1). 
    \item If, after Step 2, we obtain $u_{1}\Delta^{\pm 1}$ where $u_{1} \in \overline{(3)}$, then apply (R4) or (R5) to obtain an element $u_{1} \in \overline{(3)}$.
\end{enumerate}
By Proposition \ref{prop:odd rules invariant}, $u$ is twisted conjugate to the output $u'$ of this sequence of rewrite rules. Since (R1)-(R5) either preserves or decreases word length, then $l(u') \leq l(u)$.
\end{proof}
\subsection{Twisted cyclic geodesics}
We now consider which elements of Type $\overline{(3)}$, as defined in Definition \ref{defn:odd Type 3}, are minimal length up to twisted cyclic permutations.

\begin{defi}   
    Let $G = G(m)$ be a group with presentation as in \cref{eqn: m odd}. We define the set of \emph{twisted cyclic geodesics}, denoted $\mathsf{CycGeo}_{\phi}(G,X)$, to be the set of all geodesics $u \in X^{\ast}$, such that every $\phi$-cyclic permutation of $u$ is also geodesic.
\end{defi}
Let $\mathsf{CycGeo}_{\phi, \overline{(3)}}(G,X) = \mathsf{CycGeo}_{\phi}(G,X) \cap \; \text{Type} \; \overline{(3)}$ be the set of all twisted cyclic geodesics of Type $\overline{(3)}$. The following example highlights that when we apply a $\phi$-cyclic permutation to $u \in \overline{(3)}$, we can obtain a word which is no longer geodesic.
\begin{exa}
    Suppose $m = 5$, and consider $u = y^{-3}x^{-1} \in X^{\ast}$, which is geodesic of length 4. We can take a twisted cyclic permutation of $u$ by applying the $\phi$-cyclic permutation $u = y^{-3}x^{-1} \LR \phi\left(x^{-1}\right)y^{-3} = xy^{-3} = u'.$ Here $u'$ is not geodesic, since $u' =_{G} x^{-1}y^{2}$, which is of length 3.
\end{exa}
To consider when words of Type $\overline{(3)}$ lie in $\mathsf{CycGeo}_{\phi, \overline{(3)}}(G,X)$, we first consider when words of Type (3) are geodesic.
\begin{defi}
For $u = x^{a_{1}}y^{b_{1}}\dots x^{a_{\tau}}y^{b_{\tau}} \in (3)$, define
\begin{alignat*}{2}
    \Pos_{x}(u) &:= \mathrm{max}\{a_{i} \; : \; a_{i} \geq 0, 1 \leq i \leq \tau \}, \quad \Neg_{x}(u) &&:= \mathrm{max}\{-a_{i} \; : \; a_{i} \leq 0, 1 \leq i \leq \tau \}, \\
    \Pos_{y}(u) &:= \mathrm{max}\{b_{i} \; : \; b_{i} \geq 0, 1 \leq i \leq \tau \}, \quad \Neg_{y}(u) &&:= \mathrm{max}\{-b_{i} \; : \; b_{i} \leq 0, 1 \leq i \leq \tau \}.
\end{alignat*} 
\end{defi}

\begin{prop}\label{prop:geods}\cite[Proposition 3.5]{Fujii2018}
Let $u \in (3)$ be geodesic. Then the following conditions must all be satisfied:
\begin{multicols}{2}
    \begin{enumerate}
    \item $\Pos_{x}(u) + \Neg_{x}(u) \leq 2$.
    \item $\Pos_{y}(u) + \Neg_{y}(u) \leq 2k+1$.
    \item $\Pos_{x}(u) + \Neg_{y}(u) \leq k+1$.
    \item $\Pos_{y}(u) + \Neg_{x}(u) \leq k+1$.
\end{enumerate}
\end{multicols}
\end{prop}
Suppose $u \in \overline{(3)}$ is geodesic and $u \LR u'$ for some $u' \in X^{\ast}$. We want to establish when $u'$ is geodesic, and hence $u \in \mathsf{CycGeo}_{\phi, \overline{(3)}}(G,X)$. Let $u = x^{a_{1}}y^{b_{1}}\dots x^{a_{\tau}}y^{b_{\tau}} \in \overline{(3)}$ and consider a $\phi$-cyclic permutation of the form
\[ u' = x^{-a_{i}}y^{-b_{i}}\dots x^{-a_{\tau}}y^{-b_{\tau}}x^{a_{1}}y^{b_{1}}\dots x^{a_{i-1}}y^{b_{i-1}},
\]
for some $1 \leq i \leq \tau$. We consider each of the 4 conditions from Proposition \ref{prop:geods}, when applied to $u'$. 

The first condition is immediate, since $\Pos_{x}(w) \leq 1$ and $\Neg_{x}(w) \leq 1$ for all words $w \in X^{\ast}$. 
Since $u'$ contains at least one $x^{\pm 1}$ term, then by Conditions 3 and 4, there cannot be any $y^{\pm (k+1)}$ terms in $u'$. Condition 2 always holds with this restriction. This immediately implies that for $u'$ to be geodesic, then $u$ cannot be of Type $(3^{0+}N)$ or $(3^{0-}N)$ (recall \cref{tab:geo forms odd}), since both of these require the existence of $y^{\pm (k+1)}$ terms. This allows us to fully classify twisted cyclic geodesics of Type $\overline{(3)}$.

\begin{cor}\label{prop:cyc geos}   
    Let $\left(3^{++}\right) = \left(3^{+}\right) \cup \left(3^{0+}U\right)$, and let $\left(3^{--}\right) = \left(3^{-}\right) \cup \left(3^{0-}U\right)$. Then 
    \[ \mathsf{CycGeo}_{\phi, \overline{(3)}}(G,X) = \{ u \in \left(3^{++}\right) \cup \left(3^{--}\right) \cup \left(3^{0*}\right) \mid u \; \text{starts and ends with opposite letters} \}.
    \]
    In particular, if $u = x^{a_{1}}y^{b_{1}}\dots x^{a_{\tau}}y^{b_{\tau}} \in \mathsf{CycGeo}_{\phi, \overline{(3)}}(G,X)$, then $u$ satisfies the following conditions:
    \begin{enumerate}
        \item[(i)] $a_{1} = 0$ if and only if $b_{\tau} = 0$, 
        \item[(ii)] $-k \leq b_{i} \leq k \; (1 \leq i \leq \tau), \; b_{i} \neq 0 \; (1 \leq i \leq \tau-1)$, and
        \item[(iii)] $\begin{cases}
            0 \leq a_{i} \leq 1 \; (1 \leq i \leq \tau),  & a_{i} \neq 0 \; (2 \leq i \leq \tau), \; u \in (3^{++}), \\
            -1 \leq a_{i} \leq 0 \; (1 \leq i \leq \tau),  & a_{i} \neq 0 \; (2 \leq i \leq \tau), \; u \in (3^{--}), \\
             -1 \leq a_{i} \leq 1 \; (1 \leq i \leq \tau),  & a_{i} \neq 0 \; (2 \leq i \leq \tau), \; u \in (3^{0*}).
        \end{cases}$
    \end{enumerate}
\end{cor}
For notation, we let $\mathsf{CycGeo}_{\phi, \overline{(3)}} = \mathsf{CycGeo}_{\phi, \overline{(3)}}(G,X)$. 
\begin{lem}\label{cor:odd cyc geo obtained}
    Let $u \in \overline{(3)}$. Then there exists $u' \in \mathsf{CycGeo}_{\phi, \overline{(3)}}$ such that $u \sim_{\phi} u'$ and $l(u) \geq l(u')$. 
\end{lem}
\begin{proof}
    By the previous discussion, we only need to consider words $u \in \overline{(3)}$ which contain any $y^{\pm (k+1)}$ terms. For $\epsilon \in \{ \pm 1 \}$, we can first rewrite all pairs $\left(y^{\e(k+1)}, y^{-\epsilon(k+1)}\right)$ as $\left(y^{-\epsilon k}, y^{\epsilon k}\right)$, until we are left only with $y^{k+1}$ terms or $y^{-(k+1)}$ terms (see Proposition \ref{prop:linear time geos}). Then, we can take a $\phi$-cyclic permutation, which switches all $y^{\epsilon(k+1)}$ terms to $y^{-\epsilon(k+1)}$, and then apply the following rewrite rule to obtain a word $u' \in \mathsf{CycGeo}_{\phi, \overline{(3)}}$:
    \begin{equation}\label{eqn:x y pair rewrite}
        sx^{\epsilon}ty^{-\epsilon(k+1)}z =_{G} sx^{2\epsilon}x^{-\epsilon}ty^{-\epsilon(2k+1)}y^{\epsilon k}z =_{G} sx^{-\epsilon}ty^{\epsilon k}z,
    \end{equation}
    where $s,t,z \in X^{\ast}$.
\end{proof}
We also have to consider when non-unique geodesics can occur in $\mathsf{CycGeo}_{\phi, \overline{(3)}}$. 
\begin{prop}\label{prop:non unique}\cite[Proposition 3.8]{Fujii2018}\\
    Let $u \in (3)$ be geodesic. Then $u$ is a non-unique geodesic representative if at least one of the following conditions is satisfied:
\begin{enumerate}
    \item There exists both $x$ and $x^{-1}$ terms in $u$, or
    \item $\Pos_{y}(u) + \Neg_{y}(u) = m$.
\end{enumerate}
\end{prop}
\begin{cor}\label{rmk:unique geos}
    Let $u \in \mathsf{CycGeo}_{\phi, \overline{(3)}}$ be a non-unique geodesic. Then the only rewrite rule we can apply to $u$ is that of the form
    \begin{equation}\label{eqn: odd parity switch}
        sx^{\epsilon}tx^{-\epsilon}z =_{G} sx^{2\epsilon}x^{-\epsilon}tx^{-2\epsilon}x^{\epsilon}z  =_{G} sx^{-\epsilon}tx^{\epsilon}z,
    \end{equation}
    where $\epsilon \in \{ \pm 1 \}, s,t,z \in X^{\ast}$.
\end{cor}
\begin{proof}
    For any $y^{b}$ factor in $u$, we have that $|b| \leq k$ by Corollary \ref{prop:cyc geos}, and so Condition 2 of Proposition \ref{prop:non unique} is never satisfied. For the first condition, the only rewrite rule we can apply is \cref{eqn: odd parity switch}, for all pairs of $\left(x,x^{-1}\right)$ terms. 
\end{proof}
\begin{rem}\label{rmk:equality check}
    We note that if $u \in \mathsf{CycGeo}_{\phi, \overline{(3)}}$, then $u$ is a non-unique geodesic precisely when $u \in \left(3^{0*}\right)$. \cref{eqn: odd parity switch} allows us to check quickly if two geodesics $u,v \in \mathsf{CycGeo}_{\phi, \overline{(3)}}$ represent the same element. In particular, if $u = x^{a_{1}}y^{b_{1}}\dots x^{a_{\tau_{u}}}y^{b_{\tau_{u}}}, \; v = x^{\alpha_{1}}y^{\beta_{1}}\dots x^{\alpha_{\tau_{v}}}y^{\beta_{\tau_{v}}} \in \mathsf{CycGeo}_{\phi, \overline{(3)}}$, then $u=_{G} v$ if and only if $\tau_{u}=\tau_{v}$, $b_{i} = \beta_{i}$ for all $1 \leq i \le \tau_{u}$, and the sets $X_{u} = \{a_{1}, \dots, a_{\tau_{u}}\}$ and $X_{v} = \{\alpha_{1}, \dots, \alpha_{\tau_{u}}\}$ are equal. 
\end{rem}
 The following result highlights that if $u \in \mathsf{CycGeo}_{\phi, \overline{(3)}}$ is a non-unique geodesic, then any equivalent word $u'=_{G} u$ is also of minimal length up to twisted cyclic permutations.
\begin{lem}\label{lem:cyc geos}
    Let $u \in \mathsf{CycGeo}_{\phi, \overline{(3)}}$, and suppose $u =_{G} u'$ for some geodesic $u' \in X^{\ast}$. Then $u' \in \mathsf{CycGeo}_{\phi, \overline{(3)}}$.
\end{lem}

\begin{proof}
    If $u$ is uniquely geodesic we are done, so suppose $u$ is a non-unique geodesic. From Corollary \ref{rmk:unique geos}, we only need to consider rewriting of the form in \cref{eqn: odd parity switch}. Assume $u$ is of the form $u = sxtx^{-1}z$, for some $s,t,z \in X^{\ast}$, and $u' = sx^{-1}txz =_{G} u$. We now check if all $\phi$-cyclic permutations of $u'$ are geodesic. All $\phi$-cyclic permutations of $u'$ are equivalent to a $\phi$-cyclic permutation of $u$, which are geodesic by assumption, except for the following cases.
    \begin{enumerate}
        \item[(i)] $u' = sx^{-1}txz \LR x^{-1}\phi(z)sx^{-1}t = v_{1}$. The only possibility for $v_{1}$ to not be geodesic is if there exists a $y^{\beta}$ factor in either $\phi(z), s$ or $t$, where $\beta = \pm (k+1)$. This however is impossible since $u \in \mathsf{CycGeo}_{\phi, \overline{(3)}}$, so no $y^{\beta}$ terms in $s,t$, and $z$ can equal $y^{\pm (k+1)}$.
    \item[(ii)] $u' = sx^{-1}txz  \LR txz\phi(s)x = v_{2}$. A similar argument to (i) shows $v_{2}$ must also be geodesic.
    \end{enumerate}
    A symmetric argument holds in the case where $u = sx^{-1}txz$. 
\end{proof}
With this set $ \mathsf{CycGeo}_{\phi, \overline{(3)}}$ of minimal length twisted conjugacy representatives, we can show that a finite sequence of $\phi$-cyclic permutations and equivalent geodesics exists between any two twisted conjugate elements which lie in $\mathsf{CycGeo}_{\phi, \overline{(3)}}$. 
\begin{thm}\label{thm:main}
    Let $u,v \in \mathsf{CycGeo}_{\phi, \overline{(3)}}$. Then $u \sim_{\phi} v$ if and only if $l(u) = l(v)$ and there exists a finite sequence $u = u_{0}, u_{1}, \dots, u_{s} = v$, such that $u_{i} \leftrightarrow u_{i+1}$ via a $\phi$-cyclic permutation or an equivalent geodesic, for all $0 \leq i \leq s-1$.
\end{thm}
\begin{proof}
The reverse implication is clear by Definition \ref{defn: odd cyclic perms} and Lemma \ref{lem:cyc geos}. For the forward direction, assume $v =_{G} \phi(w)^{-1}uw = \mathsf{rev}(w)uw$ for some $w \in \X$. We assume $w$ is written in geodesic form, as defined in \cref{tab:geo forms odd}. We consider different cases for $w \in X^{\ast}$.
\par 
\textbf{Case 1: $w = \Delta^{c}$}.\\
Suppose $v =_{G} \mathsf{rev}(\Delta^{c})u\Delta^{c} =_{G} u\Delta^{2c}$. Without loss of generality, assume $c>0$. Since $v \in \mathsf{CycGeo}_{\phi, \overline{(3)}}$, $v$ is Garside free, so there exists free cancellation between $u$ and $\Delta^{2c}$. First, note that no cancellation of $\Delta^{2c}$ can occur with any $y$ factors in $u$. If this were the case, we'd have
\[ v =_{G} u\Delta^{2c} = sy^{b_{i}}t\Delta^{2c} =_{G} sy^{b_{i}+m}y^{-m}t\Delta^{2c} =_{G} sy^{b_{i}+m}t\Delta^{2c-1},
\]
for some $b_{i}$ where $-k \leq b_{i} < 0$, and some $s,t \in X^{\ast}$. Since $b_{i}+m \geq k+1$, and no $y^{\beta}$ term, where $|\beta| > k$, can occur in $v$ by Corollary \ref{prop:cyc geos}, then either this type of cancellation cannot occur, or we need to reverse this rewrite. Since $c>0$, the only way we can do this is through the existence of an $x^{-1}$ term in $u$. If this is the case, we can apply the following rewrite: 
\begin{equation}\label{eqn:odd proof seq}
    v =_{G} u\Delta^{2c} = sx^{-1}ty^{b_{i}}z\Delta^{2c} =_{G} sx^{-1}ty^{b_{i}+m}z\Delta^{2c-1} =_{G} sx^{-2}xty^{b_{i}}y^{m}z\Delta^{2c-1}  =_{G} sxty^{b_{i}}z\Delta^{2c-1},
\end{equation}
where $s,t,z \in X^{\ast}$. From \cref{eqn:odd proof seq}, we note that we can skip the middle step, and lower the power of the Garside element using the $x^{-1}$ term directly, whilst leaving all $y$ terms unchanged. Therefore, we can assume that no $y$ factors cancel with $\Delta^{2c}$, and that there exists $2c$ $x^{-1}$ terms in $u$, which cancel with $\Delta^{2c}$, to give a Garside free element as required.  
\par 
Let $u_{x} = \{a_{1}, \dots, a_{\tau}\}$ denote the multiset of exponents of $x$ factors in $u$. Similarly let $v_{x}$ be the multiset of exponents of $x$ factors in $v$. By the previous observation, there exists a (possibly non-unique) set $I = \{a_{i_{1}}, \dots, a_{i_{2c}}\} \subseteq u_{x}$, such that $v_{x} = \phi(I) \sqcup u_{x} \setminus I$, where $\phi(I) = \{-a_{i_{1}}, \dots, -a_{i_{2c}}\}$. In other words, we take a subset $I \subseteq u_{x}$, such that all exponents $a_{i} \in I$ correspond to $x^{-1}$ terms which cancel with $\Delta^{2c}$, by repeatedly applying $x^{-1}\Delta = x$. In particular, all values $a_{p} \in I$ equal -1, and all values $a_{q} \in \phi(I)$ equal 1 (for $i_{1} \leq p,q \leq i_{2c})$. We can write $I = I_{1} \sqcup I_{2}$ as two disjoint sets of equal size:
\[ I_{1} = \{a_{i_{1}}, \dots, a_{i_{c}}\}, \quad I_{2} = \{a_{i_{c+1}}, \dots, a_{i_{2c}}\}.
\]
Let $u =_{G} u_{1}u_{2}$ such that $u_{1}, u_{2} \in X^{\ast}$ are geodesics, such that all $x$ factors with exponents in $I_{1}$ lie in $u_{1}$, and similarly $x$ factors with exponents in $I_{2}$ lie in $u_{2}$. By applying $x^{-1}\Delta = x$ to all $x$ factors in $u$ with exponents in $I$, we can write $v$ in the form $v =_{G} u_{1}u_{2}\Delta^{2c} =_{G} v_{1}v_{2}$, for some geodesic words $v_{1}, v_{2} \in X^{\ast}$ such that $v_{1}v_{2} \in X^{\ast}$ is also geodesic. Note $u_{1}u_{2}, v_{1}v_{2} \in \mathsf{CycGeo}_{\phi, \overline{(3)}}$ by Lemma \ref{lem:cyc geos}. We now have 
\[ u =_{G} \underbrace{u_{1}}_{I_{1}}\underbrace{u_{2}}_{I_{2}}, \quad v =_{G} \underbrace{v_{1}}_{\phi(I_{1})}\underbrace{v_{2}}_{\phi(I_{2})}.
\]
Here $v_{1}$ contains all $x$ factors with exponents from $\phi(I_{1})$, that is all $x$ factors in $u_{1}$ with exponents in $I_{1}$, which have now been rewritten. Similarly $v_{2}$ contains all $x$ factors with exponents from $\phi(I_{2})$, that is all $x$ factors in $u_{2}$ with exponents in $I_{2}$, which have now been rewritten. In other words, $u_{1}$ is equal to $v_{1}$ as a word, up to applying $\phi$ to all $x$ factors with exponents in $I_{1}$, and similarly $u_{2}$ is equal to $v_{2}$ as a word, up to applying $\phi$ to all $x$ factors with exponents in $I_{2}$. We will denote this property by $u_{1} \cong_{1} v_{1}$ and $u_{2} \cong_{2} v_{2}$. 

We want to show that there exists a sequence from $v$ to $u$ of $\phi$-cyclic permutations and equivalent geodesics. First, we can apply the following $\phi$-cyclic permutation to $v_{1}v_{2}$:
\[ v =_{G} \underbrace{v_{1}}_{\phi(I_{1})}\underbrace{v_{2}}_{\phi(I_{2})} \LR \underbrace{\phi(v_{2})}_{I_{2}}\underbrace{v_{1}}_{\phi(I_{1})} = v'.
\]
Here $v'$ has $c$ $x^{-1}$ terms in $\phi(v_{2})$ corresponding to $I_{2}$, and $c$ $x$ terms in $v_{1}$ corresponding to $\phi(I_{1})$. These powers can be switched with each other, by repeatedly applying \cref{eqn: odd parity switch}. Equivalently, we can apply $\phi$ to all $x$ factors with exponents corresponding to $I_{2}$ and $\phi(I_{1})$. This gives us a word of the form $\overline{\phi(v_{2})}\overline{v_{1}}$, where $\phi(v_{2}) \cong_{2} \overline{\phi(v_{2})}$ and $v_{1} \cong_{1} \overline{v_{1}}$. This leaves us with
\[ v =_{G} \underbrace{v_{1}}_{\phi(I_{1})}\underbrace{v_{2}}_{\phi(I_{2})} \LR \underbrace{\phi(v_{2})}_{I_{2}}\underbrace{v_{1}}_{\phi(I_{1})} =_{G} \underbrace{\overline{\phi(v_{2})}}_{\phi(I_{2})}\underbrace{\overline{v_{1}}}_{I_{1}}.
\]
Again note $\overline{\phi(v_{2})}\overline{v_{1}} \in \mathsf{CycGeo}_{\phi, \overline{(3)}}$ by Lemma \ref{lem:cyc geos}. Finally, we can apply another $\phi$-cyclic permutation to obtain a word equivalent to $u$:
\[ v =_{G} \underbrace{v_{1}}_{\phi(I_{1})}\underbrace{v_{2}}_{\phi(I_{2})} \LR \underbrace{\phi(v_{2})}_{I_{2}}\underbrace{v_{1}}_{\phi(I_{1})} =_{G} \underbrace{\overline{\phi(v_{2})}}_{\phi(I_{2})}\underbrace{\overline{v_{1}}}_{I_{1}} \LR \underbrace{\overline{v_{1}}}_{I_{1}}\underbrace{\overline{v_{2}}}_{I_{2}} = u_{1}u_{2} =_{G} u
\]
The last equality holds by the following observation. Since $u_{1} \cong_{1} v_{1} \cong_{1} \overline{v_{1}}$, then $u_{1} = \overline{v_{1}}$, since $\phi$ is of order 2. Similarly since $\phi(v_{2}) \cong_{2} \overline{\phi(v_{2})}$, then $v_{2} \cong_{2} \overline{v_{2}}$. Hence $u_{2} \cong_{2} v_{2} \cong_{2} \overline{v_{2}}$, and so $u_{2} = \overline{v_{2}}$, which completes this case.
\par 
\textbf{Case 2}: $w$ is Garside free.\\
Suppose $v =_{G} \mathsf{rev}(w)uw$, where $w$ is Garside free. Since $v$ starts and ends with opposite factors, there exists free cancellation in $\mathsf{rev}(w)uw$. Moreover, this cancellation occurs either in $\mathsf{rev}(w)u$ or $uw$, but not both, since $u$ starts and ends with opposite factors. Assume free cancellation occurs in $uw$. We let $u =_{G} u_{1}u_{2}, \; w =_{G} u^{-1}_{2}z.$ Then we can write $v =_{G} \mathsf{rev}(z)\mathsf{rev}\left(u^{-1}_{2}\right)u_{1}z$. Since $u_{1}u_{2} \in \mathsf{CycGeo}_{\phi, \overline{(3)}}$ by Lemma \ref{lem:cyc geos}, the subword $\mathsf{rev}\left(u^{-1}_{2}\right)u_{1} = \phi(u_{2})u_{1}$ must be geodesic. Assuming $u_{2}$ is maximal, we can assume $z$ is empty, since $v$ starts and ends with opposite factors. Therefore $v =_{G} \mathsf{rev}\left(u^{-1}_{2}\right)u_{1}$ and we obtain our required sequence of $\phi$-cyclic permutations and equivalent geodesics as follows:
\[ v =_{G} \mathsf{rev}\left(u^{-1}_{2}\right)u_{1} \LR u_{1}\phi\left(\mathsf{rev}\left(u^{-1}_{2}\right)\right) = u_{1}u_{2} =_{G} u.
\]
A similar proof holds when cancellation occurs in $\mathsf{rev}(w)u$.
\par 
\textbf{Case 3}: All remaining cases.\\
Assume $w = w_{1}\Delta^{c}$ is a geodesic of Type 1 or 2, and so $v =_{G} \mathsf{rev}(w_{1})uw_{1}\Delta^{2c}$. Without loss of generality, assume $w$ is of Type 1, that is $c>0$ and all $x$ factors in $w_{1}$ have positive powers. In this situation, there can exist either free cancellation within $\mathsf{rev}(w_{1})uw_{1}$, or we can rewrite factors of $\mathsf{rev}(w_{1})uw_{1}$ with the Garside element. We first consider free cancellation in $\mathsf{rev}(w_{1})uw_{1}$. Similar to Case 2, we assume $u =_{G} u_{1}u_{2}, \; w_{1} =_{G} u^{-1}_{2}z,$ and hence
\[ v =_{G} \mathsf{rev}(z)\mathsf{rev}\left(u^{-1}_{2}\right)u_{1}z\Delta^{2c}.
\]
If $z$ is the empty word, then $v =_{G} \mathsf{rev}\left(u^{-1}_{2}\right)u_{1}\Delta^{2c}$. Here we can apply Case 1 to obtain a sequence of $\phi$-cyclic permutations and equivalent geodesics between $v$ and $\mathsf{rev}(u^{-1}_{2})u_{1}$. Then since
$\mathsf{rev}(u^{-1}_{2})u_{1} \LR u_{1}u_{2} =_{G} u$, we have that $u$ and $v$ are related by a sequence of $\phi$-cyclic permutations and equivalent geodesics. 
\par 
Otherwise, suppose $z$ is not empty. We use similar techniques to Case 1, by considering rewriting $x$-factors from $\mathsf{rev}(z)\mathsf{rev}\left(u^{-1}_{2}\right)u_{1}z$ with the Garside element, as follows. The only rewrite rule we can apply is \cref{eqn: odd parity switch}, however we note that all $x^{-1}$ terms which cancel with $\Delta^{2c}$ must come from the subword $\mathsf{rev}\left(u^{-1}_{2}\right)u_{1}$. This is because $w$ is of Type 1, and so all $x$ factors in $z$ have a positive exponent. After cancellation with the Garside element, we can rewrite $v$ as 
\[ v =_{G} \mathsf{rev}(z)\overline{\mathsf{rev}\left(u^{-1}_{2}\right)u_{1}}z,
\]
after $2c$ $x^{-1}$ terms from $\mathsf{rev}\left(u^{-1}_{2}\right)u_{1}$ have cancelled with the Garside element. Since $v$ starts and ends with opposite letters, we can assume $v =_{G} \overline{\mathsf{rev}\left(u^{-1}_{2}\right)u_{1}}$, and from here we can apply the same proof as Case 1. A similar proof holds when free cancellation occurs in $\mathsf{rev}(w_{1})u$, or when we rewrite factors of $\mathsf{rev}(w_{1})uw_{1}$. 
\end{proof}

\begin{prop}\label{prop:odd twisted conj}
    The twisted conjugacy problem $\mathrm{TCP}_{\phi}(G(m))$, where $m$ is odd, $m \geq 3$, with respect to $\phi \colon x \mapsto x^{-1}, y \mapsto y^{-1}$, is solvable.
\end{prop}

\begin{proof}
    A summary of our algorithm is as follows.

    \textbf{Input:}
    \begin{enumerate}
        \item[(i)] $m \in \Z_{\geq 3}$ which is odd.
        \item[(ii)] Words $u,v \in X^{\ast}$ representing group elements. 
        \item[(iii)] $\phi \colon x \mapsto x^{-1}, y \mapsto y^{-1} \in \mathrm{Aut}(G(m))$.
    \end{enumerate}
    \textbf{Step 1: Geodesic form}
    \begin{adjustwidth}{1.5cm}{}
    Write $u,v$ in geodesic forms $\overline{u}, \overline{v}$ respectively, given by \cref{tab:geo forms odd}. 
    \end{adjustwidth}
    \textbf{Step 2: Simple cases}
    \begin{adjustwidth}{1.5cm}{}
        If $\overline{u}$ or $\overline{v}$ is of the form $x^{a}$ or $y^{b}$, then apply Lemma \ref{cor: odd easy case} and Theorem \ref{thm:main} to determine if $\overline{u}$ and $\overline{v}$ are twisted conjugate. In particular, $\overline{u} \sim_{\phi} \overline{v}$ if and only if:
        \[  
        \begin{cases}
            \overline{u} = x^{a}, \; \overline{v} = x^{b}, & a,b \; \text{have the same parity}, \\
            \overline{u} = y^{t+mc}, \; \overline{v} = y^{s+md}, & \{c,t\}, \{d,s\} \; \text{both have opposite parity or same parity}, \\
            \overline{u} = x^{a}, \; \overline{v} = y^{t+mc}, & a \; \text{even}, \{c,t\} \; \text{have same parity}.
        \end{cases}
        \]
        Otherwise, we can assume $\overline{u}, \overline{v}$ contain at least one $x$ and $y$ factor. 
    \end{adjustwidth}
    \textbf{Step 3: Twisted cyclic geodesics}
    \begin{adjustwidth}{1.5cm}{}
    Apply $\phi$-cyclic permutations and geodesic reductions to obtain $\overline{u}_{\min}, \overline{v}_{\min} \in \mathsf{CycGeo}_{\phi, \overline{(3)}}$, such that $\overline{u} \sim_{\phi} \overline{u}_{\min}$ and $\overline{v} \sim_{\phi} \overline{v}_{\min}$ (existence follows from Proposition \ref{prop:rewrite} and Lemma \ref{cor:odd cyc geo obtained}). If $l(\overline{u}_{\min}) \neq l(\overline{v}_{\min})$, then \textbf{Output} = \texttt{\color{blue}False} (by Theorem \ref{thm:main}).
    \end{adjustwidth}
    \textbf{Step 4: Set of representatives}
    \begin{adjustwidth}{1.5cm}{}
    Compute all possible sequences of $\phi$-cyclic permutations and equivalent words from $\overline{u}_{\min}$, to get a finite set $\mathcal{D}$ of minimal length twisted conjugacy representatives. If $\overline{v}_{\min} \in \mathcal{D}$, then \textbf{Output} = \texttt{\color{blue}True}. Otherwise, \textbf{Output} = \texttt{\color{blue}False} (by Theorem \ref{thm:main}). 
    \end{adjustwidth}
\end{proof}

\begin{exa}\label{exmp}
    We illustrate Step 4 of our algorithm, by considering the group \\$G(3) = \langle x,y \: | \; x^{2} = y^{3} \rangle$, and the word $u = x^{-1}yx^{-1}y \in X^{\ast}$. Here $u \in \mathsf{CycGeo}_{\phi, \overline{(3)}}$ by Corollary \ref{prop:cyc geos}, and \cref{fig:finite set shifts DAm} gives the finite set $\mathcal{D}$ of words obtained by computing all possible $\phi$-cyclic permutations and equivalent words. Here the arrowed lines indicate $\phi$-cyclic permutations.
    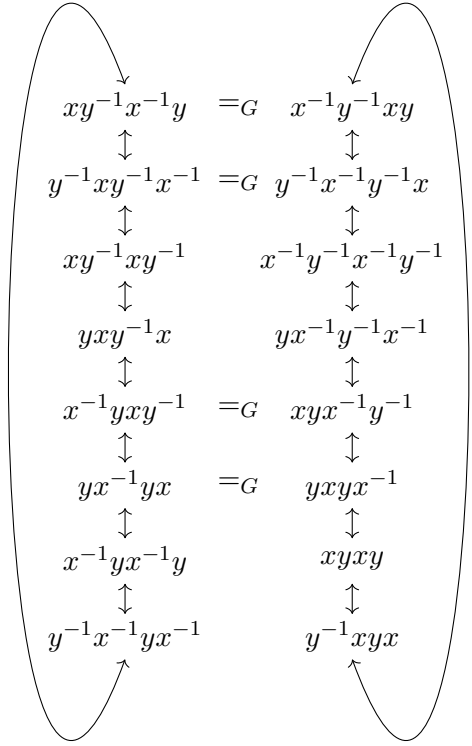
\begin{figure}[h]
        \centering
        \vspace{-125pt}
        \begin{tikzpicture}
            \node at (0,8) {$xy^{-1}x^{-1}y$};
            \node at (0,7) {$y^{-1}xy^{-1}x^{-1}$};
            \node at (0,6) {$xy^{-1}xy^{-1}$};
            \node at (0,5) {$yxy^{-1}x$};
            \node at (0,4) {$x^{-1}yxy^{-1}$};
            \node at (0,3) {$yx^{-1}yx$};
            \node at (0,2) {$x^{-1}yx^{-1}y$};
            \node at (0,1) {$y^{-1}x^{-1}yx^{-1}$};
            \node at (3,8) {$x^{-1}y^{-1}xy$};
            \node at (3,7) {$y^{-1}x^{-1}y^{-1}x$};
            \node at (3,6) {$x^{-1}y^{-1}x^{-1}y^{-1}$};
            \node at (3,5) {$yx^{-1}y^{-1}x^{-1}$};
            \node at (3,4) {$xyx^{-1}y^{-1}$};
            \node at (3,3) {$yxyx^{-1}$};
            \node at (3,2) {$xyxy$};
            \node at (3,1) {$y^{-1}xyx$};
            \node at (1.5,8) {$=_{G}$};
            \node at (1.5,7) {$=_{G}$};
            \node at (1.5,4) {$=_{G}$};
            \node at (1.5,3) {$=_{G}$};
            \draw[black, <->] (0,7.7) -- (0,7.3);
            \draw[black, <->] (0,6.7) -- (0,6.3);
            \draw[black, <->] (0,5.7) -- (0,5.3);
            \draw[black, <->] (0,4.7) -- (0,4.3);
            \draw[black, <->] (0,3.7) -- (0,3.3);
            \draw[black, <->] (0,2.7) -- (0,2.3);
            \draw[black, <->] (0,1.7) -- (0,1.3);
            \draw[black, <->] (3,7.7) -- (3,7.3);
            \draw[black, <->] (3,6.7) -- (3,6.3);
            \draw[black, <->] (3,5.7) -- (3,5.3);
            \draw[black, <->] (3,4.7) -- (3,4.3);
            \draw[black, <->] (3,3.7) -- (3,3.3);
            \draw[black, <->] (3,2.7) -- (3,2.3);
            \draw[black, <->] (3,1.7) -- (3,1.3);
            \draw[black, <->] (0, 0.7) to[out=-110,in=110, distance=6cm] (0, 8.3);
            \draw[black, <->] (3, 0.7) to[out=-70,in=70, distance=6cm] (3, 8.3);
            \end{tikzpicture}
            \vspace{-130pt}
        \caption{Set $\mathcal{D}$ obtained in Example \ref{exmp}.}
        \label{fig:finite set shifts DAm}
    \end{figure}
\end{exa}

\subsection{Complexity of algorithm}
We make some observations to improve our algorithm, which will lead to linear time complexity. We recall Example \ref{exmp} as a motivating example. Rather than computing all words of minimal length in $\mathcal{D}$, we could have instead computed all $\phi$-cyclic permutations from $u \in \mathsf{CycGeo}_{\phi, \overline{(3)}}$ only. Then, if we want to check if $v \in \mathcal{D}$, we simply check if $v$ is equal, as group elements, to a word $u' \in \mathsf{CycGeo}_{\phi, \overline{(3)}}$ obtained from $u$ via $\phi$-cyclic permutations. This step runs in linear time, based on the length of $u$, and we will go on to show that this idea can be applied in all cases. 

First, recall that for non-unique geodesics in $\mathsf{CycGeo}_{\phi, \overline{(3)}}$, the only rewrite which can occur is that of the form in \cref{eqn: odd parity switch}. This allows us to make the following observation.

\begin{prop}\label{prop:final}
    Let $u = x^{a_{1}}y^{b_{1}}\dots x^{a_{\tau}}y^{b_{\tau}}$, $v = x^{\alpha_{1}}y^{\beta_{1}}\dots x^{\alpha_{\tau}}y^{\beta_{\tau}} \in \mathsf{CycGeo}_{\phi, \overline{(3)}}$. If $u \sim_{\phi} v$, then $\left(\beta_{1}, \dots, \beta_{\tau}\right) = \left(-b_{i}, \dots, -b_{n}, b_{1}, \dots, b_{i-1}\right)$, for some $1 \leq i \leq \tau$.
\end{prop}

\begin{proof}
    By Theorem \ref{thm:main}, $u \sim_{\phi} v$ if and only if $u$ and $v$ are related by a sequence of $\phi$-cyclic permutations and equivalent geodesics. Since \cref{eqn: odd parity switch} does not affect any $y$ factors, we can assume $y$ factors are changed only by $\phi$-cyclic permutations. 
\end{proof}
Proposition \ref{prop:final} allows us to add the following information to Step 3 of our algorithm.

\textbf{Step 3: Twisted cyclic geodesics}
\begin{adjustwidth}{1.5cm}{}
Let $\overline{u}_{\min} = x^{a_{1}}y^{b_{1}}\dots x^{a_{\tau}}y^{b_{\tau}}$ and $\overline{v}_{\min} = x^{\alpha_{1}}y^{\beta_{1}}\dots x^{\alpha_{\tau}}y^{\beta_{\tau}}$. If \\ $(\beta_{1}, \dots, \beta_{\tau}) \neq (-b_{i}, \dots, -b_{\tau}, b_{1}, \dots, b_{i-1})$ for all $2 \leq i \leq \tau$, then \textbf{Output} = \texttt{\color{blue}False}.
\end{adjustwidth}
We can also adapt Step 4 of our algorithm, to find group elements of minimal length for each twisted conjugacy class, rather than words of minimal length. We find that the number of group elements of minimal length for each twisted conjugacy class is quadratic, with respect to word length.

\begin{prop}\label{prop:quadratic elements}
    Let $\overline{u}_{\min} \in \mathsf{CycGeo}_{\phi, \overline{(3)}}$, where $l(\overline{u}_{\min}) = n$. Then the number of minimal length group elements $v \in \mathsf{CycGeo}_{\phi, \overline{(3)}}$, such that $u \sim_{\phi} v$, is bounded above by $n(n+1)$. 
\end{prop}

\begin{proof}
First recall from Remark \ref{rmk:equality check} that to check if two words $u,v \in \mathsf{CycGeo}_{\phi, \overline{(3)}}$ represent the same group element, then we only need to check that the exponents of $y$ factors of $u$ and $v$ match exactly, and that the number of $x$ terms in $u$ equals the number of $x$ terms in $v$ (or analogously, the number of $x^{-1}$ terms in $u$ equals the number of $x^{-1}$ terms in $v$). Combined with Theorem \ref{thm:main}, we can find unique representatives for each group element in a twisted conjugacy class as follows.

Let $\overline{u}_{\min} = x^{a_{1}}y^{b_{1}}\dots x^{a_{m}}y^{b_{m}}$, i.e. $l(\overline{u}_{\min}) = n = m+ \sum^{m}_{i} |b_{i}|$. We can write $\overline{u}_{\min}$ in a block decomposition of the form
\[ u = \left(\# y^{b_{1}}\right)\left(\# y^{b_{2}}\right)\dots \left(\# y^{b_{m}}\right),
\]
where each $\#$ represents an $x^{\pm 1}$ term. Note if $u$ starts with a $y$ factor, then we can apply a $\phi$-cyclic permutation to obtain the correct form. Suppose the number of $x$ terms in $u$ equals $q \leq m$. Consider the following $\phi$-cyclic permutation of $u$: 
\[ u' = \left(\# y^{-b_{m}}\right)\left(\# y^{b_{1}}\right)\dots \left(\# y^{b_{m-1}}\right).
\]
Here the number of $x$ terms in $u'$ must be equal to either $q-1$ or $q+1$, depending on whether we move an $x$ or $x^{-1}$ from the last block in $u$. In particular, each time we apply a $\phi$-cyclic permutation of a block, we either increase or decrease the number of $x$ terms by 1. Moreover, the number of $x$ terms in any element is bounded below by 0 and above by $m$. 

In total, we obtain $(m+1)(m+\sum^{m}_{i}|b_{i}|)$ group elements, by computing all $\phi$-cyclic permutations of $\overline{u}_{\min}$ using this block decomposition. To see this, we group together block decomposition's which only differ by one $\phi$-cyclic permutation of a block. These pairs have either odd or even number of $x$ terms, which gives us $m+1$ group elements. Summing across all $y$ powers, to include all possible $\phi$-cyclic permutations, we get $(m+1)(m+\sum^{m}_{i}|b_{i}|)$ group elements. Since $m < n$, this value has an upper bound of $n(n+1)$.  
\end{proof}
This already improves our algorithm to $\mathcal{O}(n^{3})$, in that at Step 4, we can compute all minimal length group elements which are twisted conjugate to $\overline{u}_{\min}$, and then check if $\overline{v}_{\min}$ represents one of these group elements or not. This can be improved further to obtain a linear time Step 4 of our algorithm as follows.

\textbf{Step 4: Checking group elements}
\begin{adjustwidth}{1.5cm}{}
For $w = x^{a_{1}}y^{b_{1}}\dots x^{a_{\tau}}y^{b_{\tau}} \in \mathsf{CycGeo}_{\phi, \overline{(3)}}$, define 
\[ X(w) := 
    \begin{cases}
        (0,2,4, \dots, \tau) & \# x\text{-terms in} \; u \; \text{is even}, \\
        (1,3,5, \dots, \tau-1) & \# x \text{-terms in} \;u \; \text{is odd},
    \end{cases}
    \]
    if $\tau$ is even, and define analogously for $m$ odd. We compute all $\phi$-cyclic permutations of $\overline{u}_{\min} \mathsf{CycGeo}_{\phi, \overline{(3)}}$ using the block decomposition, whilst keeping track of how many $x$ terms can exist for each element we obtain (for each $\phi$-cyclic permutation, we switch between odd and even). For each group element $w \in \mathsf{CycGeo}_{\phi, \overline{(3)}}$ obtained via $\phi$-cyclic permutations from $\overline{u}_{\min}$, we check whether the exponents of $y$ factors of $w$ and $\overline{v}_{\min}$ match exactly, and that $X(w) = X(\overline{v}_{\min})$. If this holds for some $w$, then \textbf{Output} = \texttt{\color{blue}True}. Otherwise, if no such $w$ exists, then \textbf{Output} = \texttt{\color{blue}False} . 
\end{adjustwidth}
\begin{exa}
    Recall Example \ref{exmp}. Here $m = 2$ and $(m+1)(m+\sum^{m}_{i=1}|b_{i}|) = 3(2+2) = 12$ from Proposition \ref{prop:quadratic elements}, and so there are 12 minimal length group elements within the set $\mathcal{D}$. We list them in \cref{tab:set Dv}, by taking the minimal shortlex representative for each group element (where $x<x^{-1}<y<y^{-1})$, and also recording the parity of the number of $x$-terms in each group element.
\begin{table}[h]
    \centering
    \begin{tabular}{|c|c||c|c|}
    \hline
        Group element & Parity of $x$ terms & Group element  & Parity of $x$ terms\\
        \hline \hline
         $xy^{-1}x^{-1}y$& Odd & $y^{-1}xy^{-1}x^{-1}$ & Odd \\
         $xy^{-1}xy^{-1}$& Even  & $x^{-1}y^{-1}x^{-1}y^{-1}$  & Even \\
         $yxy^{-1}x$& Even  & $yx^{-1}y^{-1}x^{-1}$  & Even \\
         $xyx^{-1}y^{-1}$& Odd  & $yxyx^{-1}$  & Odd \\
         $x^{-1}yx^{-1}y$& Even  & $xyxy$   & Even\\
         $y^{-1}x^{-1}yx^{-1}$& Even  & $y^{-1}xyx$  & Even \\
         \hline
    \end{tabular}
    \caption{Group elements from Example \ref{exmp}}
    \label{tab:set Dv}
\end{table}

We apply our adapted Step 4 to the words $\overline{u}_{\min} = xy^{-1}x^{-1}y$, $\overline{v}_{\min} = yx^{-1}y^{-1}x^{-1}$. Here $\overline{u}_{\min} \sim_{\phi} \overline{v}_{\min}$ (see \cref{fig:finite set shifts DAm}). After computing all possible $\phi$-cyclic permutations from $\overline{u}_{\min}$, we obtain 8 words:
\begin{align*}
    &xy^{-1}x^{-1}y \leftrightarrow y^{-1}xy^{-1}x^{-1} \leftrightarrow xy^{-1}xy^{-1} \leftrightarrow yxy^{-1}x \\
    &\leftrightarrow x^{-1}yxy^{-1} \leftrightarrow yx^{-1}yx \leftrightarrow x^{-1}yx^{-1}y \leftrightarrow y^{-1}x^{-1}yx^{-1}.
\end{align*}
Consider $w = yxy^{-1}x$ from this list. Then the $y$-exponents of $w$ and $\overline{v}_{\min}$ match exactly, and $X(w) = X(\overline{v}_{\min}) = \{0,2\}$, and so our algorithm will output \texttt{\color{blue}True} as required. Alternatively, consider $\overline{v}_{\min} = yxy^{-1}x^{-1}$. Again by \cref{fig:finite set shifts DAm}, $\overline{u}_{\min} \not \sim_{\phi} \overline{v}_{\min}$. Here $\{0, 2\} = X(w) \neq X(\overline{v}_{\min}) = \{1\}$, and since $w$ is the only word from this list with matching $y$-exponents to $\overline{v}_{\min}$, then our algorithm will output \texttt{\color{blue}False} as required.
\end{exa}
We are now able to determine linear complexity for our adapted algorithm.
\begin{prop}\label{cor:odd linear time}
    The $\mathrm{TCP}_{\phi}(G(m))$, where $m$ is odd, $m \geq 3$, with respect to $\phi \colon x \mapsto x^{-1}, y \mapsto y^{-1}$, is solvable in linear time.
\end{prop}

\begin{proof}
Let $u,v \in X^{\ast}$ be our input words, where $|u| + |v| = n$. 
    Step 1 of our algorithm runs in linear time, by checking the conditions from \cref{tab:geo forms odd}, and applying the relations from $G(m)$ (see Proposition \ref{prop:linear time geos}). Linearity is also clear in Step 2.

    For Step 3, we need to analyse the complexity of applying the rewrite rules (R1)-(R5). (R1) takes at most $n-1$ time, by reading off the Garside element. To apply (R2) or (R3), we are checking the first and last terms from the free product and reducing by two each time. After all possible reductions, we have taken at most $n$ steps to do so. At this stage we may have to reapply (R1), but we only have to do this once. Finally for (R4) and (R5), we only read the first and last letters of the free product, and the Garside element of length one, so we are only reading three letters at this stage here. Overall this gives a complexity of $3n+1$. 

    Next we need to understand the complexity of finding a minimal length element in $\mathsf{CycGeo}_{\phi, \overline{(3)}}$, using Lemma \ref{cor:odd cyc geo obtained}. This only requires us to check for any $y^{\pm (k+1)}$ terms, and apply $\phi$-cyclic permutations and rewriting rules. The complexity of this step is bounded above by $2n$. 

    Finally for Step 4, we use the fact that there exists a quadratic time algorithm to check if two words are cyclic permutations of each other via a two-way deterministic pushdown automaton (2DPDA) \cite[Example 9.11]{aho_design_1974}, which can then be adapted to a linear time algorithm on a RAM machine \cite[Theorem 9.10]{aho_design_1974}. To check if $\overline{u}_{\min} \sim_{\phi} \overline{v}_{\min}$, we need to check that the $y$-terms are a $\phi$-cyclic permutation of each other, and that the differences in the $x$-terms differ by an even number. We therefore adapt our original algorithm on the 2DPDA as follows. 
    
    First, we adjust the input tape, so that we can determine if the $y$-terms are a $\phi$-cyclic permutation of each other. Then we add a counter which, when reading $x^{\pm 1}$ terms, keeps track of when they differ from the symbol on the pushdown list. The 2DPDA accepts a word if and only if the $y$-terms are a $\phi$-cyclic permutation of each other, and the counter is equal to $0 \Mod{2}$ after reading a word. This extra step of adding a counter to our 2DPDA can be achieved in constant time, and so Step 4 takes linear time overall on a RAM machine. 
\end{proof}
\comm{
\TCP

\begin{proof}
Let $u,v \in X^{\ast}$. Let $\psi \in \mathrm{Inn}(G(m))$, i.e. $\psi(w) = g^{-1}wg$ for some $g \in X^{\ast}$. By \cref{thm:odd autos}, we need to check two cases. 
\begin{enumerate}
    \item $u \sim_{\psi} v$. This case reduces to solving the conjugacy problem for $(gu, gv)$ (see \cref{rmk:outer autos enough}). The conjugacy problem in dihedral Artin groups is solvable in linear time by \cref{thm:conj prob decid DAm}.
    \item $u \sim_{\psi \circ \phi} v$, where $\phi \in \mathrm{Out}(G(m))$ is of the form in \cref{eqn: odd phi map}. This case reduces to solving the $\mathrm{TCP}_{\phi}(G(m))$ with respect to $\phi$ (see \cref{rmk:outer autos enough}). This is solvable in linear time by \cref{cor:odd linear time}.
\end{enumerate}
\end{proof}}

\section{Algorithm for the twisted conjugacy problem}\label{sec:full algorithm}
We now construct an algorithm to solve the $\mathrm{TCP}(G(m))$ for odd dihedral Artin groups. To achieve this, we first require a solution to another decision problem, known as the \emph{simultaneous conjugacy problem} (SCP).

\begin{defi}
    Let $G = \langle X \rangle$. For fixed $k \in \mathbb{N}$, the \emph{$k$-simultaneous conjugacy problem} takes as input two $k$-tuples $(y_{1}, \dots, y_{k}), (z_{1}, \dots, z_{k})$, where each $y_{i}, z_{i} \in X^{\ast}$, and decides whether there exists an element $g \in G$ such that $g^{-1}y_{i}g =_{G} z_{i}$ for all $i = 1, \dots, k$. We say this is an \emph{effective} solution if such an algorithm also produces a conjugator $g \in G$.
\end{defi}
This decision problem is also known as \emph{conjugacy for finite lists}, and some results on this decision problem can be found in \cite{kahrobaei_contracting_2024, kassabov_simultaneous_2012}. Note there exists finitely presented groups with solvable conjugacy problem and unsolvable SCP \cite{bridson_conjugacy_2005}. The following result will be useful later.

\begin{thm}\label{thm:MSCP DAM}
    The $k$-simultaneous conjugacy problem is solvable in dihedral Artin groups, for all $k \in \mathbb{N}$. Moreover, this solution is effective.
\end{thm}

\begin{proof}
    Let $G(m) = \langle X \rangle$, where $X$ is the free product generating set, and let $(y_{1}, \dots, y_{k})$, $(z_{1}, \dots, z_{k})$ be $k$-tuples where each $y_{i}, z_{i} \in X^{\ast}$. We first assume that each word $w \in X^{\ast}$ from these tuples is of the form $w = (w_{1}, w_{2})$, where $w_{1}$ is a free product normal form modulo the centre, and $w_{2}$ is a power of the central generator. 

    Let $y_{i} = (y_{i, 1}, y_{i, 2})$, $z_{i} = (z_{i, 1}, z_{i, 2})$ for some $1 \leq i \leq k$. By \cite[Proposition 3.1]{Holt2015}, checking if $y_{i}$ and $z_{i}$ are conjugate is equivalent to checking that $y_{i, 2} = z_{i, 2}$, and that $y_{i, 1}$ is a cyclic conjugate of $z_{i, 1}$. In particular, the SCP for $G(m)$ is equivalent to solving the SCP for $\Z_{2} \ast \Z_{m}$ when $m$ is odd, or $\Z_{m} \ast \Z$ when $m$ is even. In both cases, these are hyperbolic groups, which have solvable SCP by \cite[Theorem A]{bridson_conjugacy_2005}. Moreover, this solution is effective.
\end{proof}
The complexity for solving the SCP in hyperbolic groups was later improved in \cite[Theorem 1]{buckley_conjugacy_2013}. We will use these complexity estimates to determine the overall complexity of the twisted conjugacy problem in odd dihedral Artin groups.
\begin{prop}\label{prop:find g inner auto}
    Let $G(m) = \langle X \rangle$ be an odd dihedral Artin group. Let $\psi = \iota_{g} \phi \in \mathrm{Aut}(G(m))$. Then there exists an algorithm which takes as input $\psi \in \mathrm{Aut}(G(m))$, and determines the element $g \in G(m)$ defined by $\iota_{g} \in \mathrm{Inn}(G(m))$. 
    
    Moreover, the algorithm runs in time $\mathcal{O}(2\mu)$, where $\mu = 4\cdot \mathrm{max}\{1, l(\psi(x)), l(\psi(y))\}$. In particular, the algorithm runs in linear time based on the generators $x,y \in X$ and the automorphism $\psi \in \mathrm{Aut}(G(m))$ given.
\end{prop}

\begin{proof}
    By Theorem \ref{thm:odd autos}, we have two cases to consider. To check if $\phi \in [\psi]$ is trivial, we need to check if there exists $g \in G(m)$ such that $g^{-1}xg =_{G(m)} \psi(x)$ and $g^{-1}yg =_{G(m)} \psi(y)$. This is an example of the SCP in $G(m)$, for which we can determine the element $g \in G(m)$ by Theorem \ref{thm:MSCP DAM}. Otherwise, suppose $\phi \in \mathrm{Aut}(G(m))$ is of the form in \cref{eqn: odd phi map}. In this case, we need to check if there exists $g \in G(m)$ such that $g^{-1}\phi(x)g = g^{-1}x^{-1}g =_{G(m)} \psi(x)$, and $g^{-1}\phi(y)g = g^{-1}y^{-1}g =_{G(m)} \psi(y)$. Again the element $g \in G(m)$ can be determined using the SCP. 

    For complexity, the inputs of the SCP algorithm are $(x^{\e_{x}}, y^{\e_{y}}), (\psi(x), \psi(y))$, where $\e_{x}, \e_{y} \in \{ \pm 1 \}$. By \cite[Theorem 1]{buckley_conjugacy_2013}, the algorithm has complexity as stated, noting that we have to account for all 4 cases for the different values of $\e_{x}$ and $\e_{y}$.
\end{proof}

\TCP

\begin{proof}
Let $u,v \in X^{\ast}$, and let $\psi = \iota_{g}\phi \in \mathrm{Aut}(G(m))$ be our inputs. We want to decide whether there exists $w \in X^{\ast}$ such that $v =_{G(m)} \psi(w)^{-1}uw$. We will use Remark \ref{rmk:outer autos enough} to reduce the problem to twisted conjugacy with respect to outer automorphisms. Recall $\iota_{g} \in \mathrm{Inn}(G(m))$, i.e. $\iota_{g}(w) =_{G(m)} g^{-1}wg$ for some $g \in G(m)$, which can be determined in linear time by Proposition \ref{prop:find g inner auto}. We need to check two cases: 
\begin{enumerate}
    \item $u \sim_{\iota_{g}} v$: Here $v =_{G(m)} \iota_{g}(w)^{-1}uw = g^{-1}w^{-1}guw$. Rearranging gives $gv =_{G(m)} w^{-1}(gu)w$, and so this case reduces to solving the conjugacy problem with respect to $(gu, gv)$. Since $g$ is known, this is solvable in linear time by \cite[Proposition 3.1]{Holt2015}. 
    \item $u \sim_{\iota_{g} \phi} v$, where $\phi \in \mathrm{Aut}(G(m))$ is of the form in \cref{eqn: odd phi map}: Here $v =_{G(m)} g^{-1}\phi(w)^{-1}guw$. Rearranging gives $gv =_{G(m)} \phi(w)^{-1}(gu)w$, and so this case reduces to solving the twisted conjugacy problem $\mathrm{TCP}_{\phi}(G(m))$ with respect to $\phi$. This is solvable in linear time by Proposition \ref{cor:odd linear time}. 
\end{enumerate}
%\vspace{-20pt}
\end{proof}

\section{Conjugacy problem in extensions of $G(m)$}\label{sec:extension}
The aim of this section is to prove the following.
\orbit 
This will allow us to apply a criteria from \cite{bogopolski_orbit_2009} to find new examples of group extensions of odd dihedral Artin groups with solvable conjugacy problem. We recall this criteria here.
\begin{thm}\label{thm:orbit conditions}\cite[Theorem 3.1]{bogopolski_orbit_2009}
    Let 
    \[ 1 \rightarrow F \xrightarrow{\alpha} G \xrightarrow{\beta} H \rightarrow 1
    \]
    be an algorithmic short exact sequence of groups such that
    \begin{enumerate}
        \item[(i)] $F$ has solvable twisted conjugacy problem,
        \item[(ii)] $H$ has solvable conjugacy problem, and
        \item[(iii)] for every $1 \neq h \in H$, the subgroup $\langle h \rangle$ has finite index in its centralizer $C_{H}(h)$, and there is an algorithm which computes a finite set of coset representatives $z_{h,1}, \dots z_{h, t_{h}} \in H$, i.e.
        \[ C_{H}(h) = \langle h \rangle z_{h,1} \sqcup \dots \sqcup \langle h \rangle z_{h,t_{h}}.
        \]
        \end{enumerate}
        Then the conjugacy problem for $G$ is decidable if and only if the action subgroup \\ $A_{G} = \{\varphi_{g} \mid g \in G \} \leq \mathrm{Aut}(F)$ is orbit decidable.
\end{thm}

\begin{defi}
    Let $A \leq \mathrm{Aut}(G)$ for a group $G = \langle X \rangle$. The \emph{orbit decidability problem} for $A$, denoted OD(A), takes as input two words $u,v \in X^{\ast}$, and decides whether there exists $\phi \in A$ such that $v \sim \phi(u)$.  
\end{defi}
When $m$ is odd, we recall that $|\mathrm{Out}(G(m))| = 2$ by Theorem \ref{thm:odd autos}, so we can apply a similar technique as \cite[Theorem 5.1]{gonzalez-meneses_twisted_2014} to prove Theorem \ref{thm:orbit decid}.

\begin{proof}[Proof of Theorem \ref{thm:orbit decid}]
   Let $\psi_{1}, \dots \psi_{s} \in \mathrm{Aut}(G(m))$ be given, and consider $A = \langle \psi_{1}, \dots \psi_{s} \rangle \leq \mathrm{Aut}(G(m))$. For each $i = 1,\dots, s$, compute $g_{i} \in G(m)$ such that $\psi_{i} = \iota_{g_{i}}\phi_{i}$, where $\iota_{g_{i}} \in \mathrm{Inn}(G(m))$ and $\phi_{i} \in [\psi_{i}]$. Given two words $u,v \in X^{\ast}$, we want to decide whether $v \sim \psi(u)$ for some $\psi \in A$. 

If $[\psi_{i}]$ is trivial, for every $i$, then $A \leq \mathrm{Inn}(G(m))$, and so the set $\{\psi(u) \mid \psi \in A \}$ is a collection of conjugates of $u$. Here our problem reduces to deciding if $v$ is conjugate to $u$, which is decidable by \cite[Proposition 3.1]{Holt2015}. Otherwise, the set $\{\psi(u) \mid \psi \in A \}$ is a collection of conjugates of $u$ and $\phi(u)$, where $\phi \in \mathrm{Aut}(G(m))$ as in \cref{eqn: odd phi map}. We therefore need to decide whether $v$ is conjugate to either $u$ or $\phi(u)$, which is decidable by two applications of \cite[Proposition 3.1]{Holt2015}.
\end{proof}
Our final result is then immediate from Theorem \ref{thm:main result TCP solvable}, Theorem \ref{thm:orbit decid} and Theorem \ref{thm:orbit conditions}.

\extension 
\section*{Acknowledgements}
The author would like to thank Laura Ciobanu and Fujii Michihiko for helpful discussions. They would also like to thank Andrew Duncan and Alessandro Sisto for their comments, as well as the reviewer for suggestions. 

%\clearpage
\bibliography{references}
\bibliographystyle{plain}

%\begin{appendices}
\appendix
\section{Geodesic normal forms}\label{appendix}
\cref{tab:geo forms odd} gives a summary of geodesic normal forms for $G(m)$, originally derived from \cite{Fujii2018}. We provide details on how to find geodesic representatives in $G(m)$ in linear time.

\begin{prop}\label{prop:linear time geos}
    Let $w \in X^{\ast}$ represent an element of $G(m)$. Then we can determine, in linear time, a geodesic representative $u = x^{a_{1}}y^{b_{1}}\dots y^{b_{\tau}}\Delta^{c}$, where $\Delta=_{G} x^2 =_{G} y^{2k+1}$, with conditions given in \cref{tab:geo forms odd}, such that $w =_{G} u$. 
\end{prop}
\begin{proof}
    Our first step is to rewrite $w \in X^{\ast}$ into a Garside normal form. This gives us a representative $w_{1} =_{G} w$, where
    \[ w_{1} = x^{a_{1}}y^{b_{1}}\dots y^{b_{\tau}}\Delta^{c},
    \]
    such that $\tau \in \Z_{>0}$, $c \in \Z$, and
    \[
    \begin{cases}
        0 \leq a_{i} \leq 1 \; (1 \leq i \leq \tau), & a_{i} \neq 0 \; (2 \leq i \leq \tau),\\
        0 \leq b_{i} \leq 2k \; (1 \leq i \leq \tau), & b_{i} \neq 0 \; (1 \leq i \leq \tau - 1).
    \end{cases}
    \]
    This is achieved by reading $w$ once, and either applying free cancellation to each factor, or moving any $x^{\pm 2}$ or $y^{\pm (2k+1)}$ terms to the Garside element. Then, we read the word again, and for every $x^{a_{i}}$ and $y^{b_{i}}$ term, such that $a_{i}, b_{i} <0$, we replace these terms by $x^{a_{i}+2}\Delta^{-1}$ and $y^{b_{i}+2k+1}\Delta^{-1}$ respectively, and move all $\Delta^{-1}$ terms to the Garside element. 

    Next, we rewrite $w_{1}$ into a modified Garside normal form. This gives us a representative $w_{2} =_{G} w$, where
    \[ w_{2} = x^{a_{1}}y^{b_{1}}\dots y^{b_{\tau}}\Delta^{c},
    \]
    such that $\tau \in \Z_{>0}$, $c \in \Z$, and
    \[
    \begin{cases}
        0 \leq a_{i} \leq 1 \; (1 \leq i \leq \tau), & a_{i} \neq 0 \; (2 \leq i \leq \tau),\\
        -(k-1) \leq b_{i} \leq k+1 \; (1 \leq i \leq \tau), & b_{i} \neq 0 \; (1 \leq i \leq \tau - 1).
    \end{cases}
    \]
    This is achieved by reading $w_{1}$, and for every $y^{b_{i}}$ term such that $b_{i} \geq k+2$, we replace each $y^{b_{i}}$ term with $y^{b_{i}-(2k+1)}\Delta$, and move all $\Delta$ terms to the Garside element. If $c \geq 0$, then $w_{2}$ is a geodesic of Type $(1)$,  $(3^{+})$ or $(3^{+} \cap 3^{-})$. For the remaining cases, we define the following quantities based on our modified normal form $w_{2}$. Let
    \[ \begin{cases}
        R_{a} := \{i \mid a_{i} > 0 \}, & r_{a} := |R_{a}|, \\
        R_{b} := \{j \mid k \leq b_{j} \leq k+1 \}, & r_{b} := |R_{b}|, \\
        r_{w_{2}} := r_{a} + r_{b}. & 
    \end{cases}
    \]
    If $c < 0$ and $r_{w_{2}} \leq -c$, then replace all $x^{a_{i}}$, such that $i \in R_{a}$, by $x^{a_{i}-2}\Delta$, and similarly replace all $y^{b_{j}}$, such that $j \in R_{b}$, by $y^{b_{j}-(2k+1)}\Delta$, and move all $\Delta$ terms to the Garside. This gives us a geodesic of Type (2) or $(3^{-})$. If $c \leq 0$ and $r_{w_{2}} > -c$, then first rewrite all terms as before from $R_{a}$ and $R_{b}$, until the Garside element is empty. This leaves us with a word $w_{3} =_{G} w$, where
    \[ w_{3} = x^{a_{1}}y^{b_{1}}\dots y^{b_{\tau}},
    \]
    such that
    \[
    \begin{cases}
        -1 \leq a_{i} \leq 1 \; (1 \leq i \leq \tau), & a_{i} \neq 0 \; (2 \leq i \leq \tau),\\
        -(k+1) \leq b_{i} \leq k+1 \; (1 \leq i \leq \tau), & b_{i} \neq 0 \; (1 \leq i \leq \tau - 1).
    \end{cases}
    \]
    If $w_{3} \in (3^{0+}U) \cup (3^{0-}U) \cup (3^{0+}N) \cup (3^{0-}N) \cup (3^{0*})$ then we are done. Otherwise, we read $w_{3}$ and determine the following values:
    \begin{alignat*}{2}
        \alpha^{+} &:= \# x \; \text{terms in} \; w_{3}, \quad &&\alpha^{-} := \# x^{-1} \; \text{terms in} \; w_{3}, \\
        \beta^{+} &:= \# y^{k+1} \; \text{terms in} \; w_{3}, \quad &&\beta^{-} := \# y^{-(k+1)} \; \text{terms in} \; w_{3}.
    \end{alignat*}
    We note that for $\epsilon \in \{ \pm 1 \}$, then any pairs $\left(y^{\epsilon(k+1)}, y^{-\epsilon(k+1)}\right)$ can be rewritten as $\left(y^{-\epsilon k}, y^{\epsilon k}\right)$, using the rewrite rule
    \[ sy^{\epsilon (k+1)}ty^{-\epsilon (k+1)}z =_{G} sy^{\epsilon(2k+1)}y^{-\epsilon k}ty^{-\epsilon(2k+1)}y^{\epsilon k}z =_{G} sy^{-\epsilon k}ty^{\epsilon k}z,
    \]
    where $s,t,z \in X^{\ast}$. After rewriting all possible pairs of this form in $w_{3}$, we have a word where either $\beta^{+} = 0$ or $\beta^{-} = 0$. If $\beta^{+} = \beta^{-} = 0$, then our word is of Type $(3^{0+}U)$, $(3^{0-}U)$ or $(3^{0*})$.  

    Suppose $\beta^{+} = 0$ and $\beta^{-} > 0$. If $\alpha^{+} = 0$, then $w_{3}$ is of Type $(3^{0-}N)$. Otherwise, if $\alpha^{+} > 0$, then rewrite all pairs $\left(x, y^{-(k+1)}\right)$ to $(x^{-1}, y^{k})$ as far as possible, using the rewrite rule defined in \cref{eqn:x y pair rewrite}. This leaves us with a word of Type $(3^{0-}N)$ or $(3^{0*})$. Now suppose $\beta^{+}>0$ and $\beta^{-} = 0$. If $\alpha^{-} = 0$, then $w_{3}$ is of Type $(3^{0+}N)$. Otherwise, if $\alpha^{-} > 0$, then rewrite all pairs $(x^{-1}, y^{k+1})$ to $(x, y^{-k})$ as far as possible, again using \cref{eqn:x y pair rewrite}, which leaves us with a word of Type $(3^{0+}N)$ or $(3^{0*})$. 

    This covers all cases, and we obtain a geodesic representative in linear time.
\end{proof}
%\newpage
\begin{table}[ht!]
\begin{adjustbox}{width=1\textwidth}
    \begin{tabular}{|c|c|c|}
        \hline
         Type & Conditions on $w = x^{a_{1}}y^{b_{1}}\dots x^{a_{\tau}}y^{b_{\tau}}\Delta^{c}$, where $\tau \in \Z_{>0}$ & Unique/Non-unique \\
        \hline \hline
         $(1)$ & $\begin{array}{cc}
        c>0, &\\
        0 \leq a_{i} \leq 1 \; (1 \leq i \leq \tau), & a_{i} \neq 0 \; (2 \leq i \leq \tau),\\
        -(k-1) \leq b_{i} \leq k+1 \; (1 \leq i \leq \tau), & b_{i} \neq 0 \; (1 \leq i \leq \tau - 1).
        \end{array}$ & \multirow{7}{*}{Unique}\\ \cline{1-2}
        $(2)$ & $\begin{array}{cc}
            c<0, & \\
            -1 \leq a_{i} \leq 0 \; (1 \leq i \leq \tau), & a_{i} \neq 0 \; (2 \leq i \leq \tau), \\
            -(k+1) \leq b_{i} \leq k-1 \; (1 \leq i \leq \tau), & b_{i} \neq 0 \; (1 \leq i \leq \tau - 1).
        \end{array}$ & \\\cline{1-2}
        $(3^{+})$ & $\begin{array}{cc}
        c=0, & \\
        0 \leq a_{i} \leq 1 \; (1 \leq i \leq \tau), & a_{i} \neq 0 \; (2 \leq i \leq \tau),\\
        -(k-1) \leq b_{i} \leq k+1 \; (1 \leq i \leq \tau), & b_{i} \neq 0 \; (1 \leq i \leq \tau - 1).
    \end{array}$ & \\\cline{1-2}
        $(3^{-})$ & $\begin{array}{cc}
        c=0, & \\
        -1 \leq a_{i} \leq 0 \; (1 \leq i \leq \tau), & a_{i} \neq 0 \; (2 \leq i \leq \tau), \\
            -(k+1) \leq b_{i} \leq k-1 \; (1 \leq i \leq \tau), & b_{i} \neq 0 \; (1 \leq i \leq \tau - 1).
    \end{array}$ & \\\cline{1-2}
         $(3^{+}\cap 3^{-})$ & $w = y^{b}$ where $-(k-1) \leq b \leq k-1$.  & \\\cline{1-2}
        $(3^{0+}U)$ & $\begin{array}{cc}
        c = 0 & \\
        0 \leq a_{i} \leq 1 \; (1 \leq i \leq \tau), & a_{i} \neq 0 \; (2 \leq i \leq \tau),\\ 
        -k \leq b_{i} \leq k \; (1 \leq i \leq \tau), & b_{i} \neq 0 \; (1 \leq i \leq \tau -1), \\ \text{There exist at least one} \; y^{-k} \; \text{term}. & \\
    \end{array}$&  \\\cline{1-2}
         $(3^{0-}U)$ & $\begin{array}{cc}
        c = 0 & \\
        -1 \leq a_{i} \leq 0 \; (1 \leq i \leq \tau), & a_{i} \neq 0 \; (2 \leq i \leq \tau),\\ 
        -k \leq b_{i} \leq k \; (1 \leq i \leq \tau), & b_{i} \neq 0 \; (1 \leq i \leq \tau -1), \\ \text{There exist at least one} \; y^{k} \; \text{term}. & \\
    \end{array}$ & \\
        \hline
         $(3^{0+}N)$ & $\begin{array}{cc}
        c = 0 & \\
        0 \leq a_{i} \leq 1 \; (1 \leq i \leq \tau), & a_{i} \neq 0 \; (2 \leq i \leq \tau),\\ 
        -k \leq b_{i} \leq k+1 \; (1 \leq i \leq \tau), & b_{i} \neq 0 \; (1 \leq i \leq \tau -1), \\ 
        \text{There exist both} \; y^{-k} \; \text{and} \; y^{k+1} \; \text{terms}. & \\
    \end{array}$ & \multirow{3}{*}{Non-unique} \\\cline{1-2}
        $(3^{0-}N)$ & $\begin{array}{cc}
        c = 0 & \\
        -1 \leq a_{i} \leq 0 \; (1 \leq i \leq \tau), & a_{i} \neq 0 \; (2 \leq i \leq \tau),\\ 
        -(k+1) \leq b_{i} \leq k \; (1 \leq i \leq \tau), & b_{i} \neq 0 \; (1 \leq i \leq \tau -1), \\ \text{There exist both} \; y^{k} \; \text{and} \; y^{-(k+1)} \; \text{terms}. &
    \end{array}$ & \\\cline{1-2}
         $(3^{0*})$ & $\begin{array}{cc}
        c = 0 & \\
        -1 \leq a_{i} \leq 1 \; (1 \leq i \leq \tau), & a_{i} \neq 0 \; (2 \leq i \leq \tau), \\ 
        -k \leq b_{i} \leq k \; (1 \leq i \leq \tau), & b_{i} \neq 0 \; (1 \leq i \leq \tau -1), \\
        \text{There exist both} \; x, x^{-1} \; \text{terms.} &
    \end{array}$ &   \\
    \hline
    \end{tabular}
    \end{adjustbox}
    \caption{Geodesic normal forms}
    \label{tab:geo forms odd}
\end{table}
%\end{appendices}

\comm{
\uppercase{\footnotesize{Department of Mathematics, University of Manchester M13 9PL, UK and the Heilbronn Institute for Mathematical Research, Bristol, UK}}
\par 
\textit{Email address:} \texttt{gemma.crowe@manchester.ac.uk}}

\end{document}